\documentclass[reqno,12pt]{amsart}
\usepackage{amsmath, amsfonts, amsthm, mathtools, extarrows,amssymb,mathrsfs,dsfont}
\usepackage[x11names]{xcolor}
\usepackage[colorlinks=true, linkcolor=Brown4, citecolor=Chocolate4,pagebackref]{hyperref}
\usepackage{amsrefs}
\usepackage{cleveref}
\usepackage{thmtools}
\usepackage{tikz-cd}
\usepackage[shortlabels]{enumitem}
\setlist[enumerate]{leftmargin=*}
\usepackage[margin=1in,a4paper]{geometry}

\usepackage[indent]{parskip}

\title[Finiteness of pointed families of symplectic varieties]{Finiteness of pointed families of symplectic varieties: a geometric Shafarevich conjecture}

\author[L. Fu]{Lie Fu}
\address{Universit\'e de Strasbourg, IRMA, Strasbourg, France}
\email{lie.fu@math.unistra.fr}

\author[Z. Li]{Zhiyuan Li}\address{Shanghai Center for Mathematical Sciences, Fudan University, 2005 Songhu Road 200438, Shanghai, China}\email{zhiyuan\_li@fudan.edu.cn}

\author[T. Takamatsu]{Teppei Takamatsu}
\address{Department of Mathematics, Faculty of Science,
Saitama University,
255 Shimo-Okubo, Sakura-ku,
Saitama-shi, Saitama 338-8570,
Japan}
\email{teppeitakamatsu.math@gmail.com}

\author[H. Zou]{Haitao Zou}
\address{Faculty of Mathematics, Universität Bielefeld \\
Universitätsstraße 25, 33615 Bielefeld, Germany}
\email{hzou@math.uni-bielefeld.de}


\numberwithin{equation}{section}
\theoremstyle{plain}
\newtheorem{theorem}{Theorem}[section]
\newtheorem{proposition}[theorem]{Proposition}

\newtheorem{corollary}[theorem]{Corollary}
\newtheorem{lemma}[theorem]{Lemma}
\newtheorem*{step}{Step}
\newtheorem*{Shafproblem}{Pointed Shafarevich problem}

\theoremstyle{definition}
\newtheorem{definition}[theorem]{Definition}
\newtheorem{question}[theorem]{Question}
\newtheorem{conjecture}[theorem]{Conjecture}
\newtheorem{example}[theorem]{Example}
\newtheorem{remark}[theorem]{Remark}

\theoremstyle{remark}

\DeclareMathOperator{\Aut}{Aut}

\DeclareMathOperator{\Spec}{Spec}

\DeclareMathOperator{\Sp}{sp}

\DeclareMathOperator{\End}{End}

\DeclareMathOperator{\hdg}{Hdg}
\DeclareMathOperator{\id}{id}

\DeclareMathOperator{\Sym}{Sym}
\DeclareMathOperator{\Sch}{Sch}
\DeclareMathOperator{\Hilb}{Hilb}
\DeclareMathOperator{\Grps}{Grpoids}

\DeclareMathOperator{\reg}{reg}
\DeclareMathOperator{\sing}{sing}
\DeclareMathOperator{\codim}{codim}

\DeclareMathOperator{\tf}{tf}
\DeclareMathOperator{\rk}{rk}
\DeclareMathOperator{\Mon}{Mon}

\DeclareMathOperator{\an}{an}
\DeclareMathOperator{\Div}{div}
\DeclareMathOperator{\prim}{prim}
\DeclareMathOperator{\tr}{tr}

\newcommand{\Cl}{{\rm Cl}}
\newcommand{\Pic}{{\rm Pic}}
\newcommand{\KS}{{\rm KS}}
\newcommand{\NS}{{\rm NS}}
\newcommand{\Gal}{{\rm Gal}}
\newcommand{\GL}{{\rm GL}}
\newcommand{\SO}{{\rm SO}}
\newcommand{\PGL}{{\rm PGL}}

\newcommand{\Nef}{{\rm Nef}}
\newcommand{\Amp}{{\rm Amp}}
\newcommand{\BA}{{\rm BA}}
\newcommand{\Bir}{{\rm Bir}}
\newcommand{\Mov}{{\rm Mov}}
\DeclareMathOperator{\Eff}{Eff}
\DeclareMathOperator{\Conv}{Conv}
\DeclareMathOperator{\et}{\acute{e}t}

\newcommand{\Sha}{\mathrm{Shaf}}

\newcommand{\bA}{\mathbf{A}}
\newcommand{\cB}{\mathcal{B}}
\newcommand{\cC}{\mathcal{C}}

\newcommand{\cP}{\mathcal{P}}
\newcommand{\cO}{\mathcal{O}}
\newcommand{\cX}{\mathcal{X}}

\newcommand{\cL}{\mathcal{L}}
\newcommand\rH{\mathrm{H}}
\newcommand{\rL}{\mathrm{L}}
\newcommand{\rV}{\mathrm{V}}

\newcommand\rT{\mathrm{T}}
\newcommand{\rO}{\mathrm{O}}
\newcommand\spin{{\bf sp}}

\newcommand{\HK}{\mathbf{M}}
\DeclareMathOperator{\lt}{lt}

\DeclareMathOperator{\Def}{Def}

\newcommand{\NN}{\mathbb{N}}
\newcommand{\CC}{\mathbb{C}}
\newcommand{\QQ}{\mathbb{Q}}
\newcommand{\RR}{\mathbb{R}}
\newcommand{\ZZ}{\mathbb{Z}}

\newcommand{\fU}{\mathfrak{U}}
\newcommand{\PP}{\mathbb{P}}
\newcommand{\Shaf}{\mathrm{Shaf}}
\DeclareMathOperator{\Uni}{H^{\sharp}}

\DeclareMathOperator{\PsAut}{PsAut}
\DeclareMathOperator{\inv}{inv}

\DeclarePairedDelimiterX\Set[1]\lbrace\rbrace{\def\given{\;\delimsize\vert\;}#1}

\begin{document}

\begin{abstract}
We investigate in this paper the so-called pointed Shafarevich problem 
for families of primitive symplectic varieties. More precisely, for any fixed pointed curve $(B, 0)$ and any fixed primitive symplectic variety $X$, among all locally trivial families of $\QQ$-factorial and terminal primitive symplectic varieties over $B$ whose fiber over $0$ is isomorphic to $X$, we show that there are only finitely many isomorphism classes of generic fibers. Moreover, assuming semi-ampleness of isotropic nef divisors, which holds true for all hyper-Kähler manifolds of known deformation types, we show that there are only finitely many such \textit{projective} families up to isomorphism. These results are optimal since we can construct infinitely many pairwise non-isomorphic (not necessarily projective) families of smooth hyper-Kähler varieties over some pointed curve $(B, 0)$ such that they are all isomorphic over the punctured curve $B\backslash \{0\}$ and have isomorphic fibers over the base point $0$. 
\end{abstract}

\maketitle
\setcounter{tocdepth}{1}
\tableofcontents

\section{Introduction}
\subsection{Hyperbolicity of moduli spaces and polarized Shafarevich conjecture}
Let $K$ be a number field and $g\geq 2$ an integer. In his ICM address, Shafarevich \cite{Shafarevich-ICM} conjectured that, up to isomorphism, there are only finitely many smooth projective curves of genus $g$ over $K$ with good reduction outside a fixed finite set of finite places of $K$. This was proven by Faltings in \cite{Faltings83}. 
A geometric analogue where $K$ is replaced by $k(B)$ the function field of a smooth curve $B$ over an algebraically closed field $k$ of characteristic zero, was established earlier by Arakelov \cite{Arakelov71} and Paršin \cite{Parsin68}.

Faltings' result can be reformulated as saying that the moduli stack $\HK_g$ of genus-$g$ smooth projective curves is \emph{arithmetically hyperbolic} over $\overline{\QQ}$. 
Recall that a separated Deligne--Mumford stack $\mathbf{M}$ of finite type over a number field is called arithmetically hyperbolic over $\overline{\QQ}$, if $\mathbf{M}$ has a model $\mathcal{M}$ over some finitely generated subring $A\subset \overline{\QQ}$,  such that for any finitely generated subring $A' \subseteq \overline{\QQ}$ containing $A$,  the set $\mathcal{M}(A')$ of $A'$-integral points on $\mathcal{M}$ is finite; see \cite{J20}.

There is a geometric analogue of the arithmetic hyperbolicity.
Recall that a separated scheme or, more generally, a separated Deligne–Mumford stack \(\mathbf{M}\) of finite type over an algebraically closed field \(k\) of characteristic \(0\) is called \emph{geometrically hyperbolic} if, for any pointed smooth integral curve \((B, 0)\) defined over $k$ and any $k$-point \(x\) of \(\mathbf{M}\), there exist only finitely many morphisms \[f \colon B \to \mathbf{M}\] such that \(f(0) = x\) (see \cite[Definition 2.1]{JL23}). 

The \emph{Lang--Vojta conjecture} predicts a deep relation between arithmetic hyperbolicity and geometric hyperbolicity. Roughly speaking,
\begin{center}
    Arithmetic hyperbolic $\Leftrightarrow$ Geometric hyperbolic.
\end{center}
See Javanpeykar's survey \cite[\S 12]{J20} for an account of the progress towards this conjecture. 

The search for evidence of the Lang--Vojta conjecture leads to profound results in arithmetic geometry.  Combining some fundamental results in Hodge theory, one can show that a separated Deligne--Mumford stack \(\HK\) that admits a quasi-finite period map is geometrically hyperbolic (see \cite[Theorem 1.7]{JL23}).  In particular, the moduli spaces of polarized varieties satisfying the infinitesimal Torelli theorem are geometrically hyperbolic, e.g.~moduli spaces of abelian varieties, K3 surfaces, and irreducible symplectic (hyper-K\"ahler) varieties, with a given polarization type. The arithmetic hyperbolicity of the following moduli spaces is established, verifying the Lang--Vojta conjecture:
\begin{itemize}
    \item moduli space of polarized abelian varieties, by Faltings \cite[Theorem 3.1]{RationalPoint};
    \item moduli space of polarized K3 surfaces and irreducible symplectic varieties, by Andr\'e \cite{Andre}.
\end{itemize}
There exist many other families of (naturally polarized) varieties for which the analogue of Shafarevich's conjecture for polarized pairs has been established; see the summary in \cite[p.~2]{FLTZ22}. 

Inspired by the Lang--Vojta conjecture, Javanpeykar--Sun--Zuo proposed the \emph{pointed Shafarevich conjecture} in \cite[Conjecture 1.5]{JSZ} for polarized varieties with semiample canonical bundle.
Recently, in \cite{JavanpeykarLuSunZuo},  new finiteness results in this direction are obtained, which go beyond the situations where the infinitesimal Torelli theorem holds.
 
\subsection{Beyond moduli spaces: unpolarized Shafarevich conjecture}
From now on, we focus on abelian varieties and symplectic varieties (and their singular generalizations).  
Putting aside the interpretation using hyperbolicity of moduli spaces and going back to the original Shafarevich question, it is natural to ask about the finiteness of abelian schemes (of a fixed dimension) or symplectic varieties (of a fixed deformation type) over the base $\cO_{K,S}$, the $S$-integers in a number field $K$, without assuming the existence of a polarization with a bounded degree. This strengthening of the (polarized) Shafarevich conjecture is the so-called \emph{unpolarized Shafarevich conjecture}. As the name suggests, its extra difficulty stems from the lack of a polarization with uniformly bounded degree; hence, the class of varieties in question does not even fit into a single moduli stack of finite type. 

The unpolarized Shafarevich conjecture for abelian varieties is solved by Zarhin's trick; see \cite[Remark, Reduction 1, p168]{RationalPoint}. For K3 surfaces, it is verified by She in \cite{She} (see Takamatsu \cite{TakamatsucohomK3} for further discussions). For higher-dimensional smooth irreducible symplectic varieties of a fixed deformation type, the unpolarized Shafarevich conjecture, as well as its suitable cohomological variants, are proven in our previous work \cite{FLTZ22}.

\subsection{Geometric unpolarized Shafarevich conjecture}

Regarding the interplay between Geometry and Arithmetic as in the Lang--Vojta conjecture, it is interesting to formulate and study the geometric analogue of the unpolarized Shafarevich conjecture. As in the definition of geometric hyperbolicity, the meaningful statement is about the finiteness of families over pointed curves. We  formulate such a conjecture as follows for \textit{primitive symplectic varieties}, the singular generalizations of projective hyper-Kähler varieties.

\begin{Shafproblem}
\label{prob:Shafproblem}
Let $k$ be an algebraically closed field of characteristic $0$. For a \emph{pointed smooth integral curve} \((B,0)\)  and a primitive symplectic variety \( X \) defined over \( k \),  is the set
\begin{equation}\label{eq:pointshaf}
\tag{\textcolor{Brown4}{$\star$}}
    \Set*{\cX \xrightarrow{\pi} B \given \parbox{43ex}{ $\pi$ is a locally trivial family of primitive symplectic varieties with $\cX_0 \coloneqq \pi^{-1}(0) \cong X$} }\big/ \cong
\end{equation}
finite? Here $\cX$ is  an algebraic space.
\end{Shafproblem}

The families in \eqref{eq:pointshaf} are not required to be projective nor have a weak polarization of bounded degree. This leads to two main difficulties of this problem:
\begin{enumerate}
    \item (Unboundedness) a priori, the set \eqref{eq:pointshaf} is not parametrized by a moduli stack of finite type over $k$.
    \item (Non-separatedness) the relevant moduli problem is highly non-separated.
\end{enumerate}
We stress that (1) implies that the pointed Shafarevich problem is not a direct consequence of the (arithmetic or geometric) hyperbolicity of some moduli stacks. Moreover, (2) can indeed lead to some intrinsic infiniteness:  we construct in \Cref{prop:eg} infinitely many non-isomorphic families of smooth irreducible symplectic varieties over some pointed curve, such that they are all isomorphic over the punctured curve and have isomorphic fibers over the base point.

\subsection{Main results}
In this paper, we provide two results regarding the pointed Shafarevich problem for primitive symplectic varieties. The first one, \Cref{thm:introgenfinite}, concerns only the finiteness of the generic fibers; the second one, \Cref{thm:introisomfinite}, concerns only projective families.

\begin{theorem}\label{thm:introgenfinite}
Let $(B,0)$ be a pointed smooth connected curve defined over an algebraically closed field $k$ of characteristic 0.
Let \( X \) be a \(\mathbb{Q}\)-factorial terminal primitive symplectic variety over \( k \).  
\begin{enumerate}
    \item If \( b_2(X) \neq 4 \), then there are only finitely many isomorphism classes for the generic fibers of families in \eqref{eq:pointshaf};
    \item If \( b_2(X) = 4 \), the same finiteness holds for \emph{non-isotrivial} families in \eqref{eq:pointshaf}.
\end{enumerate}
\end{theorem}

\Cref{thm:introgenfinite} can be viewed as a geometric analogue of the (arithmetic) unpolarized Shafarevich conjecture proved in our previous work \cite{FLTZ22} (see also \cite{Andre}, \cite{She}, \cite{TakamatsucohomK3}). Moreover, it provides a generalization from hyper-K\"ahler manifolds to symplectic varieties with mild singularities. {The proof of \Cref{thm:introgenfinite} will be given in \Cref{sec:Finiteness of the generic fibers} in the refined form of \Cref{thm:FiniteGenericFibers}.}

\begin{theorem}\label{thm:introisomfinite}
Under the assumptions of Theorem \ref{thm:introgenfinite}, define
\[
\operatorname{Shaf}(B,X) \coloneqq
\Set*{\cX \xrightarrow{\pi} B \given
\parbox{40ex}{
$\pi$ is a flat and \textbf{projective} family of {$\mathbb{Q}$-factorial terminal primitive symplectic varieties} satisfying $\pi^{-1}(0) \cong X$}} \Big/ \cong.
\]
\begin{enumerate}
    \item If \( b_2(X) \neq 4 \) and for any family \(\mathcal{X} \to B \in \operatorname{Shaf}(B,X)\), all nef divisors on a very general fiber \( \mathcal{X}_b \) are semi-ample,  then the set \( \operatorname{Shaf}(B,X) \) is finite.
    \item If \( b_2(X) = 4 \), the set \( \operatorname{Shaf}(B,X) \) is finite when restricted to \emph{non-isotrivial} families.
\end{enumerate}
\end{theorem}

Since the semi-ampleness condition in \Cref{thm:introisomfinite}, also known as the SYZ conjecture for hyper-K\"ahler manifolds (see \Cref{sec:proofofKnownTypes}), has been established for smooth hyper-Kähler varieties of known deformation types, we can obtain the following unconditional result:
\begin{corollary}\label{cor:introFiniteKnonwTypes}
If $X$ is smooth of one of the known deformation types: $\operatorname{K3}^{[n]}$, $\operatorname{Kum}^{n}$, $\operatorname{OG6}$, or $\operatorname{OG10}$,  then the set \( \operatorname{Shaf}(B,X) \) is finite.
\end{corollary}

\begin{remark}
    Some comments on the conditions in \Cref{thm:introgenfinite} and \Cref{thm:introisomfinite} are in order.
\begin{enumerate}[(a),leftmargin=*]
\item If any locally trivial family of primitive symplectic varieties over $B$ admits a simultaneous $\QQ$-factorial terminalization over $B$, then the assumption that $X$ is $\QQ$-factorial terminal in \Cref{thm:introgenfinite} can be removed. By \cite[Proposition 5.22]{BL22} or more generally \cite[Corollary 2.29]{BGL22}, such simultaneous $\QQ$-factorial terminalization exists locally around $0 \in B$. 
    \item By Namikawa \cite{Namikawa06}, a flat family of primitive symplectic varieties with  $\QQ$-factorial and terminal fibers is locally trivial.
\item  By \Cref{thm:introgenfinite}, there are only finitely many birational classes of pointed locally trivial families. To prove the finiteness of isomorphism classes in \Cref{thm:introisomfinite}, we use the Kawamata–Morrison cone conjecture in the relative setting studied in  \cite{LZ22}, \cite{Li23} and \cite{HPX24}. The hypothesis of semi-ampleness comes from their results; see \Cref{subsec:RelativeCone} for details.
\end{enumerate}
\end{remark}

Motivated by the \emph{uniform Shafarevich conjecture} for families of canonically polarized varieties  (cf.~\cite[Theorem 3.1]{Caporaso02},\cite[Theorem 1.2]{Heier04}, \cite[Theorem 1.3]{Hei13},  \cite[Corollary 6.5]{KovacsLieblich}), it is natural to ask whether in \Cref{thm:introisomfinite} there is a uniform  upper bound, depending only on the topology of $(B, 0)$ and the deformation type of $X$, of the number of isomorphism classes of pointed projective families of primitive symplectic varieties. More precisely, we have the following question.
\begin{question}[Uniform boundedness]\label{conj:UniformBoundedness}
Let the notation be as in the Pointed Shafarevich Problem \eqref{eq:pointshaf}. Fix a group $\Gamma$ and fix a (locally trivial) deformation type $M$ of primitive symplectic variety. 
    Is there  an integer $N$, depending only on $\Gamma$ and $M$, such that 
    \[
    |\Sha(B,X) | \leq N,
    \]
    for any smooth pointed curve $(B,0)$ with $\pi_1(B, 0)\cong \Gamma$ and any primitive symplectic variety $X$ of the fixed deformation type $M$? 
   A weaker problem is whether there exists such a uniform bound which depends only on $\Gamma$ and $X$. 
\end{question}

\subsection{Strategy of proof and challenges}
As mentioned before, the families in \eqref{eq:pointshaf} are not assumed to be projective. {Assume that $k = \CC$}. We can endow their variations of Hodge structure with a \emph{weak polarization} (see \Cref{def:weakpolarization} for the definition).
Then, to reduce the proof to the polarized case, we need a kind of Zarhin's trick for these polarized Hodge structures of K3-type to bound the degree of weak polarizations in the \eqref{eq:pointshaf}. This will establish the finiteness of geometric isomorphism classes of the generic fiber in \Cref{thm:introgenfinite}. Our approach builds on the \emph{uniform Kuga-Satake construction} (see \Cref{thm:uniformKugaSatake}) developed in \cite{She}, \cite{OS18} and \cite{FLTZ22}. This was central to our proof of the arithmetic Shafarevich conjecture for smooth irreducible symplectic varieties in \cite{FLTZ22}. However, when adapting this to the pointed Shafarevich problem, we encounter a new obstruction: the associated uniform Kuga--Satake families for locally trivial families in \eqref{eq:pointshaf} are not pointed by a same abelian variety. This difficulty is resolved through a novel argument leveraging the finiteness of lattice embeddings (\Cref{prop:KS-finiteness}). 

In addition, to establish \Cref{thm:introgenfinite} and \Cref{thm:introisomfinite}, we require the Kawamata--Morrison conjecture for primitive symplectic varieties over a non-algebraically closed field {to deduce the finiteness of generic fibers and families}. For the smooth case, this is proven by the third author in \cite{takamatsu2022}. As a byproduct, in this paper, we prove the Kawamata--Morrison cone conjectures for $\QQ$-factorial primitive symplectic varieties (with $b_2 \geq 5$) over a non-algebraically closed field whose singular locus $X^{\sing}$ has codimension $\geq 4$ (see \Cref{thm:coneconj}).

To extend the finiteness results to singular primitive symplectic varieties, we establish several foundational results concerning the algebraic structure of their moduli spaces. For example, the Matsusaka--Mumford theorem for locally trivial families of primitive symplectic varieties (\Cref{prop:unpolarizedMatsusaka-Mumford}) over a geometric curve. Our method here relies on the existence of simultaneous resolution of singularities in a locally trivial family, which is still missing when $B$ is a general Dedekind scheme in positive or mixed characteristic. These developments constitute the technical core of our approach.

\subsection*{Conventions}
Throughout this paper, we let $k$ be an algebraically closed field of characteristic zero, unless otherwise noted. 
For an algebraic variety $X$ over $k$, we denote by $X_{\reg}$ the regular locus of $X$. For any $i\in \ZZ$, $b_i(X)$ denotes the $i$-th Betti number of $X$, that is, the dimension of the étale cohomology $\rH^i_{\et}(X,\QQ_{\ell}) $ as $\QQ_{\ell}$-vector spaces. For a morphism $X \to Y$ we use $X_{y}$ to denote the fiber at any point $y \in Y$.

\subsection*{Acknowledgment: } We thank Ariyan Javanpeykar, Chen Jiang, Christian Lehn, Ben Moonen, Long Wang, Shou Yoshikawa, and Kang Zuo for helpful discussions. The authors T.~Takamatsu and H.~Zou gratefully acknowledge the kind hospitality and support of Tokyo University of Science, where part of this work was carried out. We are also  grateful to anonymous referees for help comments.

\section{Primitive symplectic varieties}\label{sec:HK}

\subsection{Primitive symplectic varieties} 

In this section, we work more generally over a base field $k$ of characteristic zero. We recall the following basic concepts of symplectic varieties in the singular setting, due to Beauville \cite{Beau00}, Fujiki \cite{Fujiki-Katata} (orbifold case), and Bakker--Lehn \cite{BL22}.

\begin{definition}
A normal projective variety $X$ over $k$ is called a
\emph{symplectic variety}, if the regular locus $X_{\reg}$ carries a non-degenerate closed algebraic 2-form $\sigma$, such that there exists a resolution of singularities $\pi \colon Y \to X$ such that $\pi^* \sigma$ extends to an algebraic $2$-form on $Y$. 

A symplectic variety $X$ is called \emph{primitive symplectic}, if moreover
\begin{enumerate}
    \item $\rH^1(X,\cO_X) = 0$, and
    \item $\rH^0(X_{\reg}, \Omega_{X_{\reg}}^2) = k \sigma$.
\end{enumerate}

\end{definition}

A smooth projective irreducible symplectic variety, also known as projective hyper-K\"ahler variety  (see \cite{Beauville83localtorelli} and \cite{Huybrechts}), is clearly a primitive symplectic variety ((cf.\ \cite{Schwald})).

\begin{remark}
    The singularities of a primitive symplectic variety is always rational, i.e., for any resolution of singularities $g \colon Y \to X$, we have $Rg_*\cO_Y = \cO_X$; see \cite[Proposition 1.3]{Beau00} or \cite[Corollary 3.5]{BL22}. Moreover, a primitive symplectic variety $X$ has only terminal singularities if and only if $\codim X^{\sing} \geq 4$ by Namikawa \cite[Corollary 1]{Namikawa01}.
\end{remark}

\begin{example}
Moduli spaces of sheaves on K3  surfaces provide important examples of symplectic varieties; see \cite{Perego20} for more details.
 Let $(S,H)$ be a smooth polarized K3 surface. For a primitive Mukai vector $v_0 \in K_0(S)$ with $v_0^2 = k \geq 0$ (which is always even), we can consider the moduli space of $H$-semistable sheaves $M_H(S,v)$ on $S$ with Mukai vector $v=mv_0$ for some integer $m >0$. It is an irreducible projective normal variety if it is non-empty, which admits a symplectic form on the regular locus. Assume further that $H$ is $v$-generic.
    \begin{enumerate}
        \item If $k\geq 0, m=1$, then $M_H(S,v)$ is a smooth irreducible symplectic variety of dimension $k+2$.
        \item If $k=0, m \geq 2$, then $M_H(S,v) = S'^{(n)}$, the $n$-th symmetric product of a K3 surface $S'$. It is a primitive symplectic variety since it has a symplectic resolution $S'^{[n]} \to S'^{(n)}$ by the Hilbert scheme of $n$ points. However, it admits a quasi-étale covering $S'^{\times n} \to S'^{(n)}$ such that $\dim H^0(S'^{\times n},\Omega_{S'^{\times n}}^2) = n$, thus is not irreducible when $n \geq 2$ in the sense of \cite[Definition 8.16]{GKP}.
        \item If $k \geq 2,m\geq 2$, then it is a $\QQ$-factorial (see \cite{KLS06} and \cite{PR14})  primitive symplectic variety. If $k=m=2$, then it is at worst $2$-factorial (\cite[Theorem 1.1]{PR14}) and it admits a symplectic resolution, which is of OG10 deformation type. The rest cases are locally factorial (\cite[Theorem A]{KLS06}) with terminal singularities, since their singular loci have codimension $\geq 4$ (cf.\ \cite[Proposition 6.1]{KLS06}).
    \end{enumerate}
\end{example}

\begin{example}
    Recently, a new series of examples of $\QQ$-factorial terminal primitive symplectic varieties with $b_2(X) \geq 24$ is constructed in \cite{LLX24}, by compactifying the relative Jacobian fibration of universal families of cubic fivefolds containing a fixed cubic fourfold.
\end{example}

\begin{example}
Unlike in the smooth case, based on \cite{Fujiki-Katata}, Fu--Menet \cite[Section 5]{FuMenet-BettiNumber} constructed examples of primitive symplectic varieties with very small second Betti numbers (e.g.~$b_2 \leq 5$) by taking the symplectic quotients of smooth irreducible symplectic varieties; see also \cite[Example 6.1-6.3]{KL24} for a recent summary.
\end{example}

The automorphism groups of primitive symplectic varieties behave similarly as those of smooth ones:

\begin{lemma}\label{lem:finitePolarizedAut}
   Let $X$ be a primitive symplectic variety over $k$. Then
    \begin{enumerate}
        \item $\rH^0(X,T_X) =0$.
        \item For any ample line bundle $\cL$, the automorphism group scheme $\underline{\Aut}(X,\cL)$ is finite and \'{e}tale over $k$.
        \item If $X$ is terminal, then the birational automorphism group functor $\underline{\Bir}(X)$ (see \cite{Hanamura-Bir} for the precise definition) is represented by a locally of finite type group scheme, and $\dim \underline{\Bir}(X)=0$.  Moreover, $\Bir(X_L) =\underline{\Bir}(X)(L)$ is countable as a set for any field $k \subset L$.
    \end{enumerate}
\end{lemma}

\begin{proof}

The first statement follows directly from \cite[Lemma 4.6]{BL22}. 

For the second statement, choose $m$ such that $\cL^{\otimes m}$ is very ample and observe that the embedding $\varphi_{|\cL^{\otimes m}|} \colon X \hookrightarrow \PP^N$ gives rise to a closed immersion of group schemes $$\underline{\Aut}(X,\cL^{\otimes m}) \hookrightarrow \PGL_N.$$ 
Therefore, $\underline{\Aut}(X,\cL) \hookrightarrow \underline{\Aut}(X,\cL^{\otimes m})$ is a smooth linear algebraic group over $k$. Since $\dim \underline{\Aut}(X,\cL) = \dim \rH^0(X, T_X) = 0$, and $\operatorname{char}(k)=0$, the group scheme $\underline{\Aut}(X,\cL)$ must be finite and \'{e}tale over $k$.

In (3), the representability of $\underline{\Bir}(X)$ is given by \cite[Theorem (3.3)]{Hanamura-Bir}. Moreover, we can see $\dim \underline{\Bir}(X) = \rH^1(X,\cO_X) =0$ by Corollary (4.8) in loc.cit.. The group scheme $\Bir(X)$ is an open subscheme of the Hilbert scheme $\Hilb(X \times X)$, which has countably many connected components. Then we can see $\Bir(X_L) = \pi_0(\underline{\Bir}(X))$ is a countable set.
\end{proof}

\subsection{Locally trivial families of primitive symplectic varieties}

\begin{definition}
\label{def:hyperkahlerfamily}
Let $B$ be a complex variety, and $\mathcal{X}$ a complex algebraic space. 
A proper flat holomorphic morphism $\pi \colon \cX \to B$ is a \emph{family of primitive symplectic varieties over $\CC$} if all geometric fibers are primitive symplectic varieties.\\
Similarly, over an algebraically closed field $k$ of characteristic $0$, we define a family of primitive symplectic varieties as a proper flat morphism from a $k$-algebraic space to a $k$-variety with all geometric fibers primitive symplectic varieties over $k$.
\end{definition}

\begin{definition}\label{def:LocallyTrivial}
A family of primitive symplectic varieties $\pi \colon \mathcal{X} \to B$ over $k$ is called \emph{locally trivial} if for any closed point $x \in \cX$, there is an isomorphism of $\cO_B$-algebras
\begin{equation}\label{eq:etaleLocallyTrivial}
    \cO_{\cX,x}^{sh} \cong \cO_{B,\pi(x)}^{sh} \otimes_{k} \cO_{\cX_{\pi(x)},x}^{sh},
\end{equation}
where $(-)^{sh}$ denotes the strictly Henselization at the given point.
\end{definition}

Recall that a holomorphic map between complex analytic spaces $\pi \colon \mathcal{X} \to B$ is called \emph{locally trivial} (see \cite{FK87}) if for any point $t \in B$ and any point $p \in \mathcal{X}_t$, there exists an analytic open neighborhood $p \in \fU \subseteq \mathcal{X}$ {such that
$t = \pi(p) \in V \coloneqq \pi(\fU) \subseteq B$ is an analytic open neighborhood and
there is a biholomorphism  $\fU \cong V \times (\fU \cap \mathcal{X}_t)$ commuting with the projections to $V$.}

It is useful to observe that the local triviality of a family can be checked both formal locally or in the analytic category when $k = \CC$.
\begin{lemma}\label{lem:locallytrivial1}Let $\pi \colon \cX \to B$ be a family of primitive symplectic varieties over $\CC$. The following  conditions are equivalent:
\begin{enumerate}
\item The family $\pi$ is locally trivial in the sense of \Cref{def:LocallyTrivial}.
    \item The analytification $\pi^{\an} \colon \mathcal{X}^{\an} \to B^{\an}$ is locally trivial as a map between complex analytic spaces.
    \item for any $x \in \cX^{\an}$, there is an isomorphism of $\cO_B$-algebras
\begin{equation}\label{eq:formalLocalTrivial}
\widehat{\cO}_{\mathcal{X}^{\an}, x} \simeq  \widehat{\cO}_{B^{\an},\pi(x)} \otimes_{\CC} \widehat{\cO}_{\cX_{\pi(x)}^{\an}, x},
\end{equation}
where   $\widehat{(-)}$ denotes the formal completion at the given point.
\end{enumerate}
\end{lemma}
\begin{proof}
The implication $(2) \Rightarrow (3)$ follows directly by taking formal completion. The implication $(3) \Rightarrow (2)$ follows from \cite[Corollary (1.6)]{Artin68}.
 Moreover, Artin's approximation theorem (\cite[Corollary 2.6]{Artin69}) implies that (1) and (3) are equivalent.
\end{proof}

\begin{remark}
    In \cite{BGL22}, a locally trivial algebraic family means locally trivial in the Zariski topology, which is more restrictive than our definition here. We thank Prof.~Christian Lehn for clarifying this to us.
\end{remark}

The following fact allows us to relate the study of a locally trivial family of primitive symplectic varieties to that of a smooth family of symplectic varieties.
\begin{proposition}\label{prop:SimultaneousResolution}
Let $B$ be a complex variety,
and $\pi \colon \mathcal{X} \rightarrow B$ a locally trivial family of primitive symplectic varieties.
    Then there exists a simultaneous resolution of singularities
    \[
    \begin{tikzcd}
        \widetilde{\mathcal{X}} \ar[rr,"g"] \ar[rd] & & \mathcal{X} \ar[ld, "\pi"] \\
        & B
    \end{tikzcd}
    \]
    such that $g$ restricts to an isomorphism on $g^{-1}({\mathcal{X}}_{\reg})$ to $\mathcal{X}_{\reg}$.
\end{proposition}
\begin{proof}
This is essentially addressed in \cite[Lemma 4.9]{BL22}. Note that in loc.~cit., only $g^{\an}$ is constructed. However, the resolution $g^{\an}$ arises from an algorithmic resolution of singularities through successive blow-ups, guided by the Bierstone–Milman invariant $\inv_{b} \colon \mathcal{X}_b \to \Gamma$ at each step. Consequently, the successive blow-up loci can be globally glued as algebraic subspaces of $\mathcal{X}$ since $\inv_b(x)$ is determined by the complete local ring $\widehat{\cO}_{\cX_b,x}$. Thus, we obtain the simultaneous resolution $\widetilde{\mathcal{X}} \to B$ in the category of algebraic spaces over $B$.
\end{proof}

It is worth noting that a locally trivial family is locally acyclic in the étale topology.
\begin{lemma}
\label{lem:localconstantlocalTrivial}
If $\pi$ is locally trivial (in the sense of \Cref{def:LocallyTrivial}), then the higher direct image $R^i\pi_*^{\et} \underline{\ZZ/n}$ is locally constant in the étale topology of $B$ for any integer $i$ and positive integer $n$.
\end{lemma}

\begin{proof}
Since $R^i\pi_*^{\et}\underline{\ZZ/n}$ are constructible by the proper base change theorem, the statement is equivalent to say that the specialization map
\[
sp_{\bar{s},\bar{t}} \colon (R^i\pi_*^{\et}\underline{\ZZ/n})_{\bar{s}} \to (R^i\pi_*^{\et}\underline{\ZZ/n})_{\bar{t}}
\]
is an isomorphism for any specialization $\bar{t} \rightsquigarrow \bar{s}$ in $B$. According to \cite[Tag 0GJW]{stacks-project}, it is sufficient to show $\pi$ is locally acyclic, i.e., the pull-back
\begin{equation}\label{eq:VanishingCycle}
R\Gamma\left(\Spec(\cO_{\cX,\bar{x}}^{sh}), \underline{\ZZ/n}\right) \xrightarrow{} R\Gamma\left(\Spec(\cO_{\cX,\bar{x}}^{sh})\times_{\cO_{B,\pi(\bar{x})}^{sh}} \bar{t}, \underline{\ZZ/n} \right)
\end{equation}
is an isomorphism for any geometric point $\bar{x}$ of $\cX$, and geometric point $\bar{t}$ of $\cO_{B,f(\bar{x})}^{sh}$. Since $f$ is locally trivial, there is an isomorphism of $\cO_B$-algebras
\[
\cO_{\cX,\bar{x}}^{sh} \cong  \cO_{\cX_{\pi(\bar{x})},\bar{x}}^{sh} \otimes_{k} \cO_{B,\pi(\bar{x})}^{sh},
\]
where $(-)^{sh}$ is the strictly Henselization.
Then we have
\begin{equation}\label{eq:Kunneth}
\begin{aligned}
R\Gamma\left(\Spec(\cO_{\cX,\bar{x}}^{sh}), \underline{\ZZ/n}\right) &\cong R\Gamma\left(\Spec(\cO_{\cX_{\pi(\bar{x})},\bar{x}}^{sh} \otimes_{\CC} \cO_{B,\pi(\bar{x})}^{sh}), \underline{\ZZ/n} \right) \\ &\cong R\Gamma\left(\Spec(\cO_{\cX_{\pi(\bar{x})},\bar{x}}^{sh}),\underline{\ZZ/n}\right),
\end{aligned}
\end{equation}
where the second isomorphism is given by the Künneth formula. For any geometric point $\bar{t}$ of $\Spec(\cO_{B,f(\bar{x})}^{sh})$, the fiber product 
\[
\Spec(\cO_{\cX,\bar{x}}^{sh})\times_{\cO_{B,f(\bar{x})}^{sh}} \bar{t} \cong \Spec(\cO_{\cX_{f(\bar{x})},\bar{x}}^{sh}) \times_{\Spec(\CC)}\bar{t}.
\]
is the base extension of $\Spec(\cO_{\cX_{f(\bar{x})},\bar{x}}^{sh})$ along $k(\bar{t})/k$. Hence
\begin{equation}\label{eq:FieldExtension}
R\Gamma\left(\Spec(\cO_{\cX_{f(\bar{x})},\bar{x}}^{sh}),\underline{\ZZ/n}\right) \xrightarrow{\sim}R\Gamma\left(\Spec(\cO_{\cX,\bar{x}}^{sh})\times_{\cO_{B,f(\bar{x})}^{sh}} \bar{t}, \underline{\ZZ/n}) \right)
\end{equation}
by the geometric invariance (smooth base change) of étale cohomology (see \cite[Tag 0F0B]{stacks-project}). By the construction, the composition of \eqref{eq:Kunneth} and \eqref{eq:FieldExtension} is the pull-back \eqref{eq:VanishingCycle}. This implies $f$ is locally acyclic in the étale topology. 
\end{proof}

\subsection{{$\QQ$-factoriality in locally trivial families}}
In analyzing the birational geometry of primitive symplectic varieties, it is common to assume they are $\QQ$-factorial. However, some basic geometric operations can disrupt $\QQ$-factoriality. This section outlines essential facts for subsequent discussions. As before, $k$ is denoted for an algebraically closed field of characteristic $0$.

First, we record a technical lemma, which allows us to compare the geometric generic fiber with a very general fiber in a given family.
\begin{lemma}
\label{lem:generalpointPicard}
Let $B$ be an integral variety over $k$, and $\mathcal{X} \rightarrow B$ be a flat proper algebraic space.
Let $\Omega$ be a universal domain containing $k$, {i.e., an algebraically closed field with infinite transcendental degree over its prime field.}
Let $\eta$ be the generic point of $B$.
There exist countably many non-empty open subschemes $\Set{U_i \subset B}_{i\in I}$ such that for any $\Omega$-point
\[
b \in \bigcap_{i \in I} U_i (\Omega) \subset B (\Omega)
\]
we have an (abstract) isomorphism $\cX_{\overline{\eta_{{\Omega}}}} \cong \mathcal{X}_b$ of schemes, (which induces an isomorphism $\overline{\eta_{\Omega}} \simeq \Omega$), 
where $\overline{\eta_{\Omega}}$ is the geometric generic point of $B_{\Omega}$. In particular, there is an isomorphism
\[
\Pic(\mathcal{X}_{\overline{\eta_{{ \Omega}}}}) \simeq \Pic  (\mathcal{X}_{b}).
\]
\end{lemma}

\begin{proof}
 Set $B_{\Omega} = B \times_k \Omega$.
By the spreading out argument for morphisms of finite type, between algebraic spaces of finite type over $\Omega$, which follows from \cite[06G4, 07SK, 0CPC, 0CPE]{stacks-project}, there exist a finitely generated subfield $F \subset \Omega$ and a variety $T$ of finite type over $F$, satisfying the following.
\begin{enumerate}
\item
We have a $\Omega$-isomorphism $T \times_F \Omega \simeq B_\Omega$.
    \item
There exists a proper algebraic space $\mathcal{X}'$ over $T$ such that $\mathcal{X}' \times_F \Omega \simeq \mathcal{X}_{\Omega}$ compatible with the isomorphism in (1).
\end{enumerate}
With the same argument as in the proof of \cite[Lemma 2.1]{Vial-CyclesAndFibration}, we can find the required isomorphism
\(\cX_{\overline{\eta}_{\Omega}} \cong \cX_b\)
for some $b \in B(\Omega)$ inside the intersection of a countable set of non-empty open subschemes $\{ U_i  \subset B \}_{i \in I}$.
\end{proof}

\begin{definition}
Let $\mathcal{X} \rightarrow B$ be a locally trivial family of primitive symplectic varieties over an integral variety $B$ over $k$.
We say $\mathcal{X}$ is a \emph{locally trivial family of $\QQ$-factorial primitive symplectic varieties} when $\mathcal{X}_{b}$ is $\QQ$-factorial primitive symplectic variety for any closed point $b \in B$. 
\end{definition}

In general, the notion of $\QQ$-factorial is not stable under étale base change. Fortunately, it satisfies geometric invariance under mild singularity assumptions. The following statement was taught to the third author by Shou Yoshikawa. 

\begin{lemma}
\label{lem:Q-factorialunderbasechange}
Let $F/k$ be an extension of algebraically closed fields.
Let $X$ be a normal projective variety over $k$.
Then $X_F$ is a normal projective $\QQ$-factorial klt variety if and only if $X$ is $\QQ$-factorial and klt.
The same holds if we replace klt by terminal.
\end{lemma}
\begin{proof}
Let us only show that if $X$ is $\QQ$-factorial and klt then $X_F$ is $\QQ$-factorial. The other assertions are easy and left to the reader. 
Let $Y$ be a projective log resolution of $X$.
Then, as $X$ is $\QQ$-factorial, we have a sequence of birational maps over $k$:
\[
Y = Y_{0} \overset{f_0}\dashrightarrow Y_{1} \overset{f_1}\dashrightarrow \cdots \overset{f_{n-1}}\dashrightarrow Y_{n} =X \,,
\]
by the proof of \cite[Theorem 22.1]{Lyu-Murayama}.
Here, each dotted arrow represents either a divisorial contraction or a flip (for some klt pair) over $k$.
We claim that all $Y_{i,F}$ are $\QQ$-factorial. In particular, $X_{n,F} = X_F$ is $\QQ$-factorial. The claim can be proved by induction. 
Since $Y_{0,F}$ is regular, $Y_{0,F}$ is clearly $\QQ$-factorial.
Assume that $Y_{i,F}$ is $\QQ$-factorial.
\begin{enumerate}
    \item  If $f_i$ is a divisorial contraction, then the base extension $f_{i,F} \colon Y_{i,F} \rightarrow Y_{i+1,F}$ is also a divisorial contraction since it is still of relative Picard number $1$ as $k$ and $F$ are algebraically closed. Thus $Y_{i+1,F}$ is also $\QQ$-factorial (see \cite[Lemma 19.3]{Lyu-Murayama} for example).
    \item If $f_i$ is a flip, i.e., it factors as
    \[
    \begin{tikzcd}[row sep =small]
        Y_{i} \ar[rr,dashed,"f_i"] \ar[rd,"g_i"'] & &Y_{i+1} \ar[ld,"g^+_i"]\\
        & Z &
    \end{tikzcd}
    \]
    where $g_i$ and $g_i^+$ are small contractions. After the base extension to $F$, $g_{i,F}$ and $g_{i,F}^+$ are still of relative Picard number 1 as $k$ is algebraically closed.
    Moreover, clearly $g_{i,F}$ and $g_{i,F}^+$ are small contractions.
    Therefore, $f_{i,F}$ is a filp, and $Y_{i+1,F}$ is $\QQ$-factorial (see \cite[Lemma 20.3]{Lyu-Murayama} for example). \endproof
\end{enumerate}

\end{proof}

Thanks to \Cref{lem:Q-factorialunderbasechange}, we can see the notion of locally trivial family of $\QQ$-factorial primitive varieties is well-behaved under a base-change.
\begin{proposition} \label{prop:loctrivQ-factterminal}
Let $\pi \colon \cX \to B$ be a locally trivial family of primitive symplectic varieties over $k$.  Let $0 \in B$ be a regular $k$-point. Suppose the fiber $X\coloneqq \mathcal{X}_0$ is a $\QQ$-factorial terminal primitive symplectic variety.
\begin{enumerate}
    \item 
    The geometric generic fiber $\mathcal{X}_{\overline{\eta}}$ is $\QQ$-factorial and terminal.
    \item \label{lem:loctrivQ-factterminal} The family $\pi$ is a locally trivial family of $\QQ$-factorial  primitive symplectic varieties if and only if any geometric fiber $\mathcal{X}_s$ is a $\QQ$-factorial  primitive symplectic variety.  Moreover, in any case, all geometric fibers of $\pi$ have terminal singularities.
\end{enumerate}
\end{proposition}
\begin{proof}
{For (1), we may assume that $B$ is regular by shrinking $B$ since (1) is about the geometric generic fiber and $0 \in B$ is assumed to be regular.}
By Lemma \ref{lem:Q-factorialunderbasechange} and the spreading out argument (see Lemma \ref{lem:generalpointPicard} for example), we may assume that $k$ is an algebraic closure of finitely generated field of characteristic $0$.
By choosing the embedding $k \hookrightarrow \CC$ and using Lemma \ref{lem:Q-factorialunderbasechange} again, we can reduce the problem to the case where $k=\CC$.
In this case, there exists an open subscheme $0 \in U \subset B$ such that $\mathcal{X}|_{U}$ is $\QQ$-factorial by \cite[(12.1.9)]{KollarMori92} and the Bertini theorem.
Here, note that, since $\pi$ is locally trivial, on each fiber, the singular locus has codimension $\geq 4$ by \cite[Corollary 1]{Namikawa01}.
By Lemma \ref{lem:generalpointPicard}, for general closed point $b\in U$,
we have an isomorphism (as a scheme) $\mathcal{X}_b \simeq \mathcal{X}_{\overline{\eta}}$.
Therefore, we obtain the desired result.

   For (2), it is enough to show the ``only if" part.
We fix a geometric point $s$ of $B$.
By taking the Zariski closure of the image of $s$, we can reduce this to  (1).
It finishes the proof.
\end{proof}

\subsection{Period map and Local Torelli Theorem}
For a complex primitive symplectic variety $X$, it is known that the torsion-free part of the second cohomology $$\rH^2(X, \ZZ)_{\tf}\coloneqq \rH^2(X, \ZZ)/\text{torsion}$$ carries a pure Hodge structure of weight $2$, since $X$ has at worst rational singularities. Moreover, there exists an integral quadratic form called Beauville--Bogomolov(--Fujiki--Namikawa) form 
\[
q_X \colon \rH^2(X,\ZZ)_{\tf} \to \ZZ
\]
that is compatible with the Hodge structure on $\rH^2(X,\ZZ)_{\tf}$ (see \cite[Theorem 8]{Namikawaextension} or \cite[Subsection 5.1]{BL22} for example).
This quadratic form coincides with the classical Beauville--Bogomolov form when $X$ is an irreducible symplectic manifold (up to scaling) and is also referred to as the \emph{Beauville--Bogomolov form} on $X$.

Analogous to irreducible symplectic manifolds, when $X$ varies in a locally trivial family, the Hodge structure and the Beauville--Bogomolov form on $\rH^2(X,\ZZ)_{\tf}$ forms a polarized variation of Hodge structure.

\begin{proposition}\label{prop:PropLocallyTrivial}
Let $\pi \colon \mathcal{X} \to B$ be a locally trivial family of primitive symplectic varieties over $\CC$. 
\begin{enumerate}
    \item The higher direct images $R^i\pi^{\an}_* \ZZ$ are $\ZZ$-local systems. Moreover, $R^2\pi^{\an}_*\ZZ$ is equipped with a $\ZZ$-variation of Hodge structure. 
    \item There exists a morphism of $\ZZ$-variations of Hodge structure.
    \[
    q\colon \Sym^2 R^2\pi^{\an}_* \ZZ(1) \to \ZZ,
    \]
    such that $q_{\mathcal{X}_b}$ on each fiber $\mathcal{X}_b$ is the Beauville--Bogomolov form.
\end{enumerate}
\end{proposition}
\begin{proof}
The statement (1) is due to Namikawa (\cite[p.13]{Namikawa06} (arXiv version)). See also \cite[Proposition 5.1]{Amerik-Verbitskycollections}. For other statements, see \cite[Corollary 3.5, Lemma 5.7, Lemma 4.9]{BL22}.
\end{proof}

The local Torelli theorem completes the picture:
\begin{proposition}[Local Torelli Theorem]\label{prop:localTorelliThm}
Let $X$ be a complex primitive symplectic variety,
and $\pi \colon \mathcal{X} \to \Def^{\lt}(X)$ be the Kuranishi family of locally trivial deformations of $X$ (see \cite[Subsection 4.4]{BL22}). Let $\Lambda = \rH^2(X,\ZZ)_{\tf}$ the lattice with Beauville--Bogomolov form. The period map 
\[
\Def^{\lt}(X) \to \Omega_{\Lambda} 
\]
associated with the variation of Hodge structure $R^2\pi_* \ZZ$ is a local isomorphism.
Here, $\Omega_{\Lambda}$ is the period domain
\[
\Omega_{\Lambda}
:=
\Set*{ [\sigma] \in \PP(\Lambda \otimes \CC) \given q(\sigma) =0, q(\sigma, \bar{\sigma})>0}.
\]

\end{proposition}
\begin{proof}
See \cite[Proposition 5.5]{BL22}.
\end{proof}

\section{Cone conjectures for primitive symplectic varieties}
For K3 surfaces, the study of the action of the automorphism group on the nef cone plays an important role for showing finiteness results; see \cite{Sterk}, \cite{Bright-Logan-vanLuijk}.
In this section, we recall its higher-dimensional generalizations, namely the so-called Kawamata--Morrison cone conjectures, which have been established for primitive symplectic varieties (see \cite{Markmansurvey}, \cite{Markman-Yoshioka}, \cite{AV17} for the smooth case and \cite{LMP22} for the generalization to the singular setting).

\subsection{Néron--Severi lattices and cones}\label{subsec:NS}
Let $B$ be a variety over $k$ and  $\pi \colon \cX \to B$ a family of primitive symplectic varieties (in the sense of \Cref{def:hyperkahlerfamily}).

Suppose $k = \CC$, the exponential sequence induces an exact sequence of analytic sheaves
\[
0 = R^1\pi_* \cO_{\mathcal{X}} \to \underline{\Pic}_{\mathcal{X}/B} \to R^2\pi_* \ZZ \to R^2 \pi_* \cO_{\mathcal{X}}.
\]
Since fibers $\mathcal{X}_b$ have at worst rational singularities, Du Bois--Jarraud's base-change theorem \cite[Theorem 2.62, Complement 2.62.5]{KollarGeneralType} implies that $R^i\pi_*\cO_{\mathcal{X}}$ is locally free for all integers $i \geq 0$.
Since $\rH^1(\mathcal{X}_b, \cO_{\mathcal{X}_b}) =0$ for any $b \in B$, the Picard scheme $\underline{\Pic}_{\mathcal{X}/B}$ is of dimension zero. For this reason, the identity component $\underline{\Pic}^0_{\mathcal{X}/B} $ is trivial and on the geometric fiber $\cX_{\bar{b}}$,
\[
\underline{\Pic}_{\mathcal{X}/B}(\bar{b})\cong \Pic(\mathcal{X}_{\bar{b}}) \cong \NS(\mathcal{X}_{\bar{b}})
\]
are Néron–Severi groups. Moreover, it is easy to see the Lefschetz-$(1,1)$ theorem holds by \Cref{prop:PropLocallyTrivial} (1).

\begin{definition}
    Let $X$ be a primitive symplectic variety over a subfield $k \subseteq \CC$. The \emph{Néron--Severi lattice} is the (torsion-free) Néron--Severi group 
    \(
    \NS(X_{k})_{\tf} = \Pic(X_{k})_{\tf}
    \)
    together with restriction of the Beauville--Bogomolov form along \[
    \NS(X_{k})_{\tf} =\NS(X_{\CC})_{\tf} \hookrightarrow \rH^2(X_{\CC}, \ZZ)_{\tf}.
    \]For simplicity, we denote $\NS(X_{k})$ for the Néron--Severi lattice $\NS(X_{k})_{\tf}$ in the following. 
\end{definition}
\begin{remark}
    By spreading out argument and the proof of \cite[Corollary 4.2.1]{Bindt}, there is a unique Beauville--Bogomolov form on $\NS(X_{k})_{\tf}$ for any field $k$ in characteristic zero, which is independent of the field embedding $k \hookrightarrow \CC$.
\end{remark}
In the following, \(\pi \colon \mathcal{X} \to B\) is assumed to be \emph{projective}, and each fiber \(\mathcal{X}_b\) for \(b \in B\) is a primitive symplectic variety with \(\QQ\)-factorial terminal singularities. In this case, the Picard scheme $\underline{\Pic}_{\cX/S}$ admits a global section over $B$ given by the relative ample line bundle and thus $\rk \underline{\Pic}_{\cX/B}(B) \geq 1$.

Consider the following relative Néron--Severi space on $\mathcal{X}$ (modulo $\pi$-numerical equivalence):
\[
N^1(\mathcal{X}/B)_{\RR} = \underline{\Pic}_{\cX/B}(B)/_{\equiv_\pi} \otimes_{\QQ} \RR
\]
where $\equiv_\pi$ is the numerical equivalence relation over $B$, i.e., $D_1 \equiv_\pi D_2$ if and only if $D_1 \cdot C = D_2 \cdot C$ for any curve in $\mathcal{X}$ such that $\pi(C)$ is a point.
\begin{definition}
    Let $\Eff(X/B) \subseteq N^1(\cX/B)_{\RR}$ be the cone generated by effective $\RR$-Cartier divisors over $B$.
\begin{enumerate}
    \item $\Nef^e(\cX/B) = \Nef(\cX/B) \cap \Eff(\cX/B)$ is the \emph{effective nef cone};
    \item $\overline{\Mov}^e(\cX/B) = \overline{\Mov}(\cX/B) \cap \Eff(\cX/B)$ is the \emph{effective movable cone}.
\end{enumerate}
Moreover, we also consider the following rational cones
\begin{enumerate}[resume]
    \item $\Nef^+(\cX/B) \coloneqq \Conv(\Nef(\cX/B) \cap N^1(\cX/B)_{\QQ})$.
    \item ${\Mov}^+(\cX/B) \coloneqq \Conv(\overline{\Mov}(\cX/B) \cap N^1(\cX/B)_{\QQ})$
\end{enumerate}
where $\Conv(-)$ stands for the convex hull.

For simplicity of notations, we denote
\[
\Nef^e(X) \quad (\text{resp.~}\overline{\Mov}^e(X))
\]
for the corresponding effective nef (resp.~movable) cone for simplicity when $B = \Spec(F)$ for a a field $F$.
\end{definition}

\begin{conjecture}[Kawamata--Morrison Cone Conjecture]
\label{conj:ConeConjecture}
Let \(\pi \colon \mathcal{X} \to B\) be a projective family of primitive symplectic varieties with \(\QQ\)-factorial terminal singularities.
\begin{enumerate}
    \item The effective nef cone \(\Nef^{e}(\cX/B)\) admits a rational polyhedral fundamental domain under the action of the relative automorphism group \(\Aut_{\pi}(\mathcal{X})\), which consists of automorphisms of \(\mathcal{X}\) over \(\pi\).
    \item The effective movable cone \(\overline\Mov^e(\cX/B)\) admits a rational polyhedral fundamental domain under the action of the relative pseudo-automorphism group \(\PsAut_{\pi}(\mathcal{X})\), which consists of birational automorphisms of \(\mathcal{X}\) over $B$ that are isomorphisms in codimension one. 
\end{enumerate}
\end{conjecture}

\subsection{Kawamata--Morrison cone conjectures over a field}

In this subsection, we consider the cone conjecture when $X \coloneqq \cX \to B = \Spec(F)$ over a field. Here $F$ is an arbitrary field in characteristic zero, which is not necessarily algebraically closed.
The motivation is to obtain some finiteness results on the generic fibers of families of primitive symplectic varieties.

Fix an algebraic closure $\overline{F}$ for $F$.
If $X_{\overline{F}}$ has $\QQ$-factorial terminal singularities, the \textit{birational ample cone} is defined as the union \[
\BA(X) \coloneqq \bigcup\limits_{f\in \cB_X} f^* \Amp(Y) \subseteq N^1(X)_{\RR},
\]where  
\[
\cB_X = \Set*{ f\colon X\dashrightarrow Y \given \parbox{15em}{$Y_{\overline{F}}$ is a $\QQ$-factorial terminal primitive symplectic variety over $\overline{F}$ and $f$ is birational} }.
\]
{Notice that being ample (resp.~ movable) is stable under field extensions, thus we have $\BA(X) = \BA(X_{\overline{F}}) \cap N^1(X)_{\RR}$ (resp.~ $\Mov(X) = \Mov(X_{\overline{F}}) \cap N^1(X)_{\RR}$) by the Galois descent as in \cite[Proposition 4.2.2]{takamatsu2022}. Then it follows from \cite[Proposition 5.8]{LMP22} that the closure of $\BA (X)$ is equal to the movable cone $\overline{\Mov}(X)$.}

\begin{theorem}\label{thm:coneconj}
Let $X$ be a projective primitive symplectic variety over $F$ with $b_2(X) \geq 5$ such that $X_{\overline{F} }$ is $\QQ$-factorial and terminal.  Then the cone $\Nef^+(X)$ (resp.~$\Mov^+(X)$) admits a fundamental domain $\Pi$ under the action of $\Aut(X)$ (resp.~$\Bir(X)$), which is a rational polyhedral subcone.

\end{theorem}
\begin{proof}
{
Let us only give a sketch of the proof and refer to Faucher \cite{Faucher2025} for full details\footnote{After the first version of the paper was made public, the authors were informed that full details were being carried out in the PhD thesis of Aur\'elien Faucher.}.}
When $F = \CC$, this is Theorem 1.2 of \cite{LMP22}. The Lefschetz principle ({here we use Lemma \ref{lem:Q-factorialunderbasechange}}) implies that it holds for any algebraically closed field of characteristic zero.

If $F$ is not algebraically closed, we note that the method of \cite[Theorem 4.2.7 (and Theorem 4.1.4)]{takamatsu2022} also applies to singular primitive symplectic varieties. Note that we also need the basic facts about birational cone conjecture over $\overline{F}$ in \cite{LMP22}, and the fact that a prime exceptional divisor on $X_{\overline{F}}$ is rigid (see \cite[Theorem 1.1]{KMPPzariski}, cf.\ \cite[Theorem 5.8]{Markmansurvey}). Thus, Theorem \ref{thm:coneconj} holds over an arbitrary base field $F$ of characteristic zero. 
\end{proof}

Suppose $F = \CC$. Recall that  \[
\Mon^{\lt}_{\hdg}(X_{\CC}) \subseteq O(\rH^2(X_{\CC},\ZZ),q_X)\] is the subgroup consisting of all parallel transport operators from a locally trivial family that contains $X_{\CC}$, which also preserves Hodge structures.
\begin{definition}\label{def:WallDiv}
Let $X_{\CC}$ be a complex primitive symplectic variety with $\QQ$-factorial terminal singularities. Then a Cartier divisor $D$ is called a \textit{wall divisor} or a \textit{monodromy birationally minimal} (MBM) class if $q(D) < 0$ and 
\[
\Phi(D^\perp) \cap \BA(X_\CC) = \varnothing,
\]
for any  $\Phi \in \Mon^{\lt}_{\hdg}(X_{\CC})$. We denote the set of wall divisors on $X_\CC$ by $\mathcal{W}(X_\CC)$.
\end{definition}
\begin{example}[MBM classes on $K3^{[n]}$-type manifolds]\label{eg:cone-K3^2}
If $X$ is a K3 surface, the MBM classes in $N^1(X)$ are the classes of $(-2)$-curves. 
{By the work of Hassett--Tschinkel \cite{HT10} (see also \cite[Theorem 4.1]{AV22}), if $(X,L)$ is a smooth polarized irreducible symplectic variety of $K3^{[2]}$-type, then the norm and divisibility of MBM classes $D$ in $N^1(X)$ satisfy one of the following conditions:
    \begin{enumerate}
        \item $q(D) = -10$, and $\Div(D) = 2$,
        \item $q(D) = -2$, and $\Div(D) = 1$, or
        \item $q(D) = -2$, and $\Div(D) = 2$.
    \end{enumerate}
More generally, for $K3^{[n]}$-type manifolds, information on the MBM classes can be obtained from the minimal model program of $K3^{[n]}$-type manifolds, see \cite{Mongardi15}, \cite{BayerMacri}, \cite{Bayer-Hassett-Tschinkel}.
}
\end{example}

Amerik and Verbitsky \cite{AV17} observed that the Beauville--Bogomolov squares of all primitive wall divisors on a smooth irreducible symplectic variety over $\CC$ are bounded from below. This fact is generalized to primitive symplectic varieties with $\QQ$-factorial terminal singularities in \cite[Proposition 7.7]{LMP22}. 
\begin{proposition}\label{prop:boundedMBM}
    Let $X$ be a primitive symplectic variety over $\CC$, with $\QQ$-factorial terminal singularities and $b_2(X) \geq 5$. There exists an integer $N > 0$ such that 
    \[
    q_Y(D) \geq -N,
    \]
    for any $\QQ$-factorial terminal primitive symplectic variety $Y$ which is locally trivial deformation equivalent to $X$, and any primitive wall divisor $D$ on $Y$.    
\end{proposition}
From this fact and \Cref{thm:coneconj}, we can deduce the finiteness of birational models of $X$ over a general field $k$ of characteristic $0$ (not necessarily algebraically closed).

\begin{corollary}\label{cor:finiteBir}
If $X$ is a projective $\QQ$-factorial primitive symplectic variety over $F$ with terminal singularities, and $b_2(X) \geq 5$, then up to $F$-isomorphism, there are only finitely many $\QQ$-factorial terminal $F$-birational models of $X$.
\end{corollary}
\begin{proof}
Let
\[
\Sigma = \Set*{f_*(E) \given \parbox{22em}{$E\in \NS(Y)$ is primitive and extremal in the dual cone $\Nef^{*}(Y)$, and\\ $f \colon Y \dashrightarrow X_{\overline{F}}$ such that $Y$ is $\QQ$-factorial terminal primitive symplectic variety over $\overline{F}$ } } 
\]
Let $\Pi \subseteq \Nef^+(X)$ be the fundamental domain of the $\Aut(X)$-action, which is given by \Cref{thm:coneconj}.
By the argument of \cite[Proposition 2.2.]{Markman-Yoshioka}, we can see 
\[
\Set*{D \in \Sigma \given D^{\bot} \cap \Pi \cap \cC  \neq \varnothing}
\]
is a countable set (here, $\mathcal{C} \subset N^1(X)$ is a connected component of the positive cone). 
We shall show that 
\[
\mathcal{W} \coloneqq \Set*{D \in \Sigma \given D^{\bot} \cap \Pi
\cap \mathcal{C}
\neq \varnothing}
\]
is a finite set.
When $\overline{F} = \CC$, this follows from \Cref{prop:boundedMBM} and \cite[Proposition 3.4]{Markman-Yoshioka}.
The general case follows from the Lefschetz principle.
Then we can conclude it by the proof of \cite[Theorem 4.2.7]{takamatsu2022}.
\end{proof}

\begin{remark}
    It is well-known that the nef cone conjecture of $X$ will imply that there are only finitely many birational contractions of $X$ up to $F$-isomorphism. See \cite[Proposition 5.3]{GLSW24}.
\end{remark}

\begin{proposition}
\label{prop:finitenessoftwists}
Let $F'/F$ be a finite field extension, and $X$ a projective primitive symplectic variety over $F$ with $\QQ$-factorial terminal singularities such that $b_2 (X)\geq 5$.
Then the set
\[
\operatorname{Tw}_{F'/F}(X)=\Set*{ Y \given \parbox{16em}{ $Y$ is a projective primitive symplectic variety over $F$ such that $Y_{F'} \simeq X_{F'}$} }/F\textup{-isom}
\]
is a finite set.
\end{proposition}

\begin{proof}
This follows from the same argument as in \cite[Theorem 4.3.6]{takamatsu2022} by using Lemma \ref{lem:finitePolarizedAut}, Theorem \ref{thm:coneconj}, and \cite[Theorem 6.16]{BL22}.
\end{proof}

{
\begin{proof}
We adapt the proof of \cite[Theorem 4.3.6]{takamatsu2022} to the setting of projective primitive symplectic varieties with $\mathbb{Q}$-factorial terminal singularities. The argument proceeds in three steps.

\noindent \textbf{Step 1: A uniform bound on polarization degrees.}
As in \cite[Theorem 4.3.6]{takamatsu2022}, we may first replace $F'$ by a finite Galois extension of $F$ containing it; this does not change the set of twists $\operatorname{Tw}_{F'/F}(X)$.  
Now let $Y$ be a projective primitive symplectic variety over $F$ with $Y_{F'} \simeq X_{F'}$. Then $Y_{\overline{F}} \simeq X_{\overline{F}}$, and we obtain a $\Gal (\overline{F}/F')$‑equivariant isometry of Néron–Severi lattices
\[
\Phi \colon \operatorname{NS}(Y_{\overline{F}}) \xrightarrow{\;\sim\;} \operatorname{NS}(X_{\overline{F}})
\]
preserving the Beauville–Bogomolov form.  
The Galois action on $\operatorname{NS}(X_{\overline{F}})$ factors through a finite quotient. Moreover, the set of possible $\Gal (\overline{F}/F')$‑module structures on a lattice of fixed rank and discriminant is finite.  

As $b_2(X)\geq 5$, using the boundedness of MBM classes \cite[Theorem 5.3]{AV17} together with the global Torelli theorem for primitive symplectic varieties \cite[Theorem 6.16]{BL22}, we can follow the analogue of \cite[Lemma 4.3.3]{takamatsu2022} in our setting. Consequently, there exists a positive integer $d$, depending only on the deformation type of $X$ and the lattice $\operatorname{NS}(X_{\overline{F}})$, such that every such $Y$ admits a polarization $L_Y$ with $(L_Y,L_Y)=d$.

\noindent \textbf{Step 2: Finiteness of polarizations of fixed square modulo automorphisms.}
By the cone conjecture for primitive symplectic varieties (Theorem \ref{thm:coneconj}), the action of $\operatorname{Aut}(X_{\overline{F}})$ on $\operatorname{Nef}^+(X_{\overline{F}})$ admits a rational polyhedral fundamental domain. Adapting \cite[Lemma 3.1.5]{takamatsu2022} to the Beauville–Bogomolov form, we conclude that for the fixed integer $d$ above, the set of polarizations of square $d$ on $X_{\overline{F}}$ modulo $\operatorname{Aut}(X_{\overline{F}})$ is finite. Choose representatives $M_1,\dots,M_m \in \operatorname{NS}(X_{\overline{F}})$ for these classes.

\noindent \textbf{Step 3: Finiteness of Galois twists.}
For each $i=1,\dots,m$, define
\[
T_i = \left\{ (Y,L) \;\middle|\; 
\begin{array}{l}
Y \text{ projective primitive symplectic over } F,\\[2pt]
L \text{ a polarization on } Y \text{ with } (L,L)=d,\\[2pt]
(Y,L)_{F'} \simeq (X_{F'},M_i)
\end{array}
\right\} \Big/ F\text{-isomorphism}.
\]
The construction in Step 1 gives each $Y \in \operatorname{Tw}_{F'/F}(X)$ a polarization $L_Y$ of square $d$. Sending $Y$ to $(Y,L_Y)$ defines an injection
\[
\operatorname{Tw}_{F'/F}(X) \hookrightarrow \bigsqcup_{i=1}^{m} T_i.
\]

It remains to show that each $T_i$ is finite. Assume $T_i \neq \emptyset$ and choose a basepoint $(Y_0,L_0) \in T_i$. Then $T_i$ is in bijection with the Galois cohomology set
\[
\rH^1\bigl(\operatorname{Gal}(F'/F),\; \operatorname{Aut}_{F'}(Y_0,L_0)\bigr),
\]
where $\operatorname{Aut}_{F'}(Y_0,L_0)$ is the automorphism group of the polarized variety $(Y_0,L_0)$ over $F'$. By Lemma \ref{lem:finitePolarizedAut}, $\operatorname{Aut}_{F'}(Y_0,L_0)$ is a finite group. Hence the cohomology set is finite.

Since $\operatorname{Tw}_{F'/F}(X)$ embeds into a finite disjoint union of finite sets, it is itself finite. This completes the proof.
\end{proof}

}

\noindent \textbf{Remark.} The condition $b_2(X) \geq 5$ is used to apply the cone conjecture (Theorem \ref{thm:coneconj}) and the boundedness of MBM classes, which are known for primitive symplectic varieties with this assumption.

\subsection{Consequences of cone conjectures: finiteness results}\label{subsec:RelativeCone}
In this part, we assume that $k = \overline{k}$ in characteristic zero as before.
As in the case of K3 surfaces \cite{Sterk}, cone conjectures imply some finiteness results. In higher-dimensions and in the relative setting, we have the following statement due to \cite[Theorem 1.4]{Li23} and \cite[Proposition 4.3, Corollary 4.4]{HPX24}.
\begin{proposition}\label{prop:RelFiniteness}
Let $\pi \colon \cX \to B$ be a projective morphism with connected $K$-trivial fibers, such that $\cX$ and $B$ are normal and $\QQ$-factorial klt variety over $k$.
Assume the good minimal model exists for all klt pairs of the geometric generic fiber of $\pi$, and 
\({\Mov}^+(\cX_{\eta}) \subseteq \Eff(\cX_{\eta})
\). If the action \[
\PsAut(\cX_{\eta}) \curvearrowright {\Mov}^+(\cX_{\eta})\]
admits a rational polyhedral fundamental domain, and 
 then there are only finitely many small $\QQ$-factorial modifications $\mathcal{X} \dashrightarrow \mathcal{X}'$ over $B$, up to isomorphism over $B$. 
\end{proposition}
\begin{proof}
{The statement in this proposition is stable under any algebraically closed field extension $k \subset L$.
Thus we may assume that the algebraically closed field $k$ is uncountable for simplicity.}
Under the assumptions, good minimal model exists for any klt pairs on a very general fiber. Moreover, there is a polyhedral subcone $\Pi \subseteq \Eff(\cX_{\eta})$ such that
\begin{equation}\label{eq:WeakConeConjecture}
\Mov^{\circ}(\cX_{\eta}) \subset \PsAut(\cX_{\eta}) \cdot \Pi
\end{equation}
by Looijenga's results (see \cite[Proposition 3.3]{LZ22} for example).
Then we can see there are only finitely many small $\QQ$-factorial modifications $\mathcal{X} \dashrightarrow \mathcal{X}'$ over $B$ by applying \cite[Proposition 4.3]{HPX24}.
\end{proof}

The following is a generalization of \Cref{cor:finiteBir} in the relative case.
\begin{corollary}\label{cor:finiteSQM}
Let $B$ be a normal integral $\QQ$-factorial variety over $k$.
    Let $\mathcal{X} \to B$ be a projective locally trivial family of $\QQ$-factorial primitive symplectic varieties with terminal singularities and second Betti number $b_2$.  Suppose
    \begin{enumerate}
        \item $b_2 \geq 5$, 
        \item {the total space $\cX$ is $\QQ$-factorial terminal,} and
        \item there on the geometric generic fiber $\mathcal{X}_{\overline{\eta}}$, all nef divisors are semi-ample.
    \end{enumerate}
     Then $\mathcal{X}$ has finitely many small $\QQ$-factorial modifications over $B$.
\end{corollary}
\begin{proof}
    Note that the generic fiber $\cX_{\eta}$ is a primitive symplectic variety with $\QQ$-factorial terminal singularities under the condition (2) (see also Proposition \ref{prop:loctrivQ-factterminal}). 
    Thus the pseudo-automorphism group $\PsAut(\cX_{\eta}) = \Bir(\cX_{\eta})$.
    \Cref{thm:coneconj} implies that the action
    \[
    \PsAut(\cX_{\eta}) \curvearrowright \overline{\Mov}^+(\cX_{\eta})
    \]
    admits a rational polyhedral fundamental domain when $b_2(\cX_{\bar{\eta}}) \geq 5$.   
    
    The condition (3) ensures that the good minimal model exists for the klt pairs $(\cX_{\bar{\eta}},\Delta)$ of the geometric generic fiber $\cX_{\bar{\eta}}$. Note that, Boucksom--Zariski decomposition for effective $\QQ$-Cartier divisors holds by \cite[Theorem 1.1]{KMPPzariski}. Then, under the assumption, we also have $\Mov^+(\cX_{\eta}) \subseteq \Eff(\cX_{\eta})$ by the proof of \cite[Proposition 5.5 (b1)]{HPX24}. 
    
    Thus \Cref{prop:RelFiniteness} imply that there are only finitely many small $\QQ$-factorial modifications of $\cX$ over $B$.
\end{proof}

Finally, we make a remark that the condition (2) in \Cref{cor:finiteSQM} is redundant when $B$ is regular.

\begin{lemma}\label{lemma:Total Space Qfactorial}
Let $k$ be an algebraically closed field of characteristic $0$.
Let $B$ be a smooth variety over $k$, and
$\mathcal{X} \rightarrow B$ a locally trivial family of $\QQ$-factorial terminal primitive symplectic varieties.
Then the total space $\cX$ is also $\QQ$-factorial and terminal.
\end{lemma}
\begin{proof}
    We may assume that $B$ is an affine scheme.
    Let $b\in B$ a closed point.
    Then by the proof of \cite[(12.1.9)]{KollarMori92} and the Bertini theorem, $\mathcal{X}$ is $\QQ$-factorial near $\mathcal{X}_b$.
    Since $b$ is any, $\mathcal{X}$ is $\QQ$-factorial.
    The terminality follows from the inversion of adjunction (see \cite[Chapter VI, Theorem 5.2]{Nakayama} for example).
\end{proof}

\section{Global moduli theory of primitive symplectic varieties}
In this section, we will construct the moduli stack $\HK_{2n,d}$ of locally trivial families of polarized primitive symplectic varieties of degree $d$ over an algebraically closed field $k$ of characteristic $0$. It turns out $\HK_{2n,d}$ is a Deligne--Mumford stack that is separated and of finite type over $k$ (see \Cref{thm:ModuliStack}). For the separatedness of $\HK_{2n,d}$, the key point is to establish the Matsusaka--Mumford theorem for singular symplectic varieties (\Cref{prop:unpolarizedMatsusaka-Mumford} and \Cref{cor:polarizedMatsusakaMumford}), which is originally stated for smooth families.

As an application, we will establish the following finiteness result by the geometric hyperbolicity of $\HK_{2n,d}$, using a method similar to that in \cite{FLTZ22}. 
\begin{theorem}\label{thm:polarizedfiniteness}
Let $(B,0)$ be a pointed smooth variety over $k$
with the generic point $\eta$.
Let $X$ be a projective primitive symplectic variety over $k$.
Let $d$ be a positive integer.
Then the following set is finite: 
\[
\Shaf_{B,X,d} \coloneqq
\Set*{(\cX_{\eta},L) \given \parbox{20em}{$\cX \to B$ is a locally trivial family of primitive symplectic varieties such that $\cX_0 \cong X$, $L \in \Pic_{\cX_{\eta}/\eta}(\eta)$ is a polarization of degree $d$} }/_{\simeq_{\eta}} .
\]
\end{theorem}

\subsection{Matsusaka--Mumford theorem for locally trivial families}\label{sec:MatsusakaMumford}

In this subsection, we prove a Matsusaka--Mumford type theorem for locally trivial families of (possibly singular) primitive symplectic varieties.

\begin{proposition}\label{prop:unpolarizedMatsusaka-Mumford}
Let $B$ be a variety over $k$, and $\widetilde{\mathcal{X}}, \widetilde{\mathcal{Y}}$ locally trivial families of primitive symplectic varieties over $B$. Fix a regular codimension-1 point of $B$, let $R$ be its local ring with residue field $s$ and fraction field $\eta$, and set $\mathcal{X} := \widetilde{\mathcal{X}}_R$, $\mathcal{Y} := \widetilde{\mathcal{Y}}_R$. For any birational map $f: \mathcal{X}_\eta \dashrightarrow \mathcal{Y}_\eta$ between generic fibers, there exist closed algebraic subspaces $V \subset \mathcal{X}$, $W \subset \mathcal{Y}$ with $V_s \neq \mathcal{X}_s$, $W_s \neq \mathcal{Y}_s$, such that $f$ extends to an isomorphism:
\[
\widetilde{f}: \mathcal{X} \setminus V \overset{\simeq}{\longrightarrow} \mathcal{Y} \setminus W.
\]
\end{proposition}

\begin{proof}
By Proposition \ref{prop:SimultaneousResolution}, we can take simultaneous resolutions 
\[
\mathcal{X}' \rightarrow \mathcal{X}
\quad 
\textup{and}
\quad
\mathcal{Y}' \rightarrow \mathcal{Y}.
\]
We note that we have a birational map
\[
f' \colon \mathcal{X}'_{\eta} \dashrightarrow \mathcal{Y}'_{\eta}
\]
induced by $f$, and $\mathcal{Y}'_{s}$ is non-ruled by the assumption.
Then by \cite[Theorem A.2]{FLTZ22},
there exist closed algebraic subspaces $V' \subset \mathcal{X}'$ and $W' \subset \mathcal{Y}'$ with 
$V'_{s} \neq \mathcal{X}'_{s}$ and 
$W'_{s} \neq \mathcal{Y}'_{s}$
such that $f'$ extends to
\[
\widetilde{f}' \colon  \mathcal{X}' \setminus V' \simeq \mathcal{Y}' \setminus W'.
\]
Also, by the construction of simultaneous resolutions, there exist closed subspaces $Z_{\mathcal{X}}$, $Z_{\mathcal{X}'}$, $Z_{\mathcal{Y}}$, $Z_{\mathcal{Y}'}$ on $\mathcal{X}$, $\mathcal{X}'$, $\mathcal{Y}$, $\mathcal{Y}'$ respectively such that 
\[
\mathcal{X}' \setminus Z_{\mathcal{X}'} \simeq \mathcal{X} \setminus Z_{\mathcal{X}}
\quad\textup{and}
\quad
\mathcal{Y}' \setminus Z_{\mathcal{Y}'} \simeq \mathcal{Y} \setminus Z_{\mathcal{Y}}.
\]
Let $V'_{\mathcal{X}}$ and $W'_{\mathcal{Y}}$ be the image of $V$ and $W$ in $\mathcal{X}$ and $\mathcal{Y}$ respectively.
Then by putting
\[
V := V'_{\mathcal{X}} \cup Z_{\mathcal{X}}
\quad
\textup{and}
\quad
W := W'_{\mathcal{Y}} \cup Z_{\mathcal{Y}},
\]
we obtain the assertion.
\end{proof}

As an application, we get the following result.
\begin{corollary}[Polarized Matsusaka--Mumford Theorem]\label{cor:polarizedMatsusakaMumford}
Let $\mathcal{X}, \mathcal{Y}$ be as in Proposition \ref{prop:unpolarizedMatsusaka-Mumford}, and suppose $f: \mathcal{X}_\eta \to \mathcal{Y}_\eta$ is an isomorphism. If there exist ample line bundles $\mathcal{L}$ on $\mathcal{X}$ and $\mathcal{M}$ on $\mathcal{Y}$ over $R$ with $f^*\mathcal{M}_\eta = \mathcal{L}_\eta$, then $f$ extends uniquely to a global isomorphism:
\[
\widetilde{f}: \mathcal{X} \overset{\simeq}{\longrightarrow} \mathcal{Y}.
\]
\end{corollary}

\begin{proof}
By Proposition \ref{prop:unpolarizedMatsusaka-Mumford}, $f$ extends to an isomorphism outside closed subspaces $V \subset \mathcal{X}$ and $W \subset \mathcal{Y}$. The polarization compatibility $f^*\mathcal{M}_\eta = \mathcal{L}_\eta$ forces $V = W = \emptyset$ via the ampleness propagation in \cite[Proposition 3.1.2]{Kollar}. Thus $\widetilde{f}$, as well as its inverse, is everywhere defined. 
\end{proof}

\begin{proposition}\label{prop:birationalspecialization}
Let \( B \) be a variety over \( k \), \( \widetilde{\mathcal{X}} \) a locally trivial family of primitive symplectic varieties over \( B \), and \( R \) the localization of \( B \) at a regular codimension-1 point with residue field \( s \) and fraction field \( \eta \). Set \( \mathcal{X} := \widetilde{\mathcal{X}}_R \). For any finite subgroup \( G \subset \Bir(\mathcal{X}_\eta) \):
\begin{enumerate}
    \item There exists a closed algebraic subspace \( V \subset \mathcal{X} \) with \( V_s \neq \mathcal{X}_s \), and a homomorphism 
    \[
    \psi: G \to \Aut(\mathcal{X} \setminus V)
    \]
    such that \( \psi(f)|_{(\mathcal{X} \setminus V)_\eta} = f \) for all \( f \in G \).
    \item The composition 
    \[
    \overline{\psi}: G \xrightarrow{\psi} \Aut(\mathcal{X} \setminus V) \to \Aut\big((\mathcal{X} \setminus V)_s\big)
    \]
    is injective.
\end{enumerate}
\end{proposition}

\begin{proof}
Apply Proposition \ref{prop:unpolarizedMatsusaka-Mumford} to each \( f \in G \). The simultaneous resolution of singularities (as in \cite[Lemma A.3]{FLTZ22}) ensures that the exceptional loci \( V_f \subset \mathcal{X} \) can be uniformly bounded. Taking \( V = \bigcup_{f \in G} V_f \), the \( G \)-equivariance of resolutions guarantees the homomorphism \( \psi \). Injectivity of \( \overline{\psi} \) follows from the faithfulness of specialization when \( V_s \neq \mathcal{X}_s \).
\end{proof}

\subsection{Moduli stack of polarized primitive symplectic varieties}

Let $\Sch_{\CC}$ be the site of schemes of finite type over $k$ with étale topology. Consider the following fibered category in groupoids
\[
\begin{aligned}
    \HK_{2n,d} \colon \Sch_{k} &\to \Grps \\
    B &\rightsquigarrow \Set*{(\mathcal{X} \xrightarrow{\pi} B, \lambda) \given \parbox{15em}{$\pi$ is a locally trivial family of primitive symplectic varieties of dimension $2n$ and $\lambda$ is a polarization on $\mathcal{X}$ with $(\lambda)^{2n}=d$.}},
\end{aligned}
\]where isomorphisms in the groupoids are the natural ones for pairs, and a polarization on $\mathcal{X}$ is an element $\lambda \in \Pic_{\mathcal{X}/B} (B)$ whose restriction on any geometric fiber is an ample line bundle. 
\begin{remark}
By the discussion in \Cref{subsec:NS}, over $\CC$, an (analytic) global section of the Picard scheme is equivalent to a family of integral Hodge $(1,1)$-classes of $\pi$ .
\end{remark}

\begin{theorem}\label{thm:ModuliStack}
    The moduli stack $\HK_{2n,d}$ of locally trivial families of polarized primitive symplectic varieties of degree $d$ is a Deligne--Mumford stack, which is smooth, separated, and of finite type over $k$.
\end{theorem}

\begin{proof}
Since all primitive symplectic varieties have rational singularities, Matsusaka's big theorem can be applied (see \cite[Theorem 2.4]{Matsusaka}). Thus the moduli stack $\HK^P$ of locally trivial families of polarized primitive symplectic varieties with Hilbert polynomial $P$ is a finite type algebraic stack over $k$ (cf.~\cite[Lemma 3.2.5, Lemma 3.3.6, 3.3.7]{Bindt}). \Cref{lem:finitePolarizedAut} implies that $\HK^P$ is Deligne--Mumford.
Therefore, the stack \[\HK_{2n,d} = \bigsqcup_{P \in I_{d,n}} \HK^P\]is a Deligne--Mumford stack locally of finite type over $k$, where
\[
I_{d,n} \coloneqq \Set*{P(t) =\sum_{i\geq 0}^{2n} a_i t^i \in  \QQ[t] \given \parbox{11em}{ $a_{2n} = \dfrac{d}{(2n)!}$ and $\HK^P \neq \varnothing$}}.
\]
The rest of the required properties can be verified as follows.
\begin{enumerate}
    \item 
    Since a primitive symplectic variety has only rational singularities, the Kodaria vanishing theorem holds. Thus, as $X$ is $K$-trivial, the Hilbert polynomial of $X$ with respect to an ample line bundle $\cL$ is equal to $P(t)= \dim \lvert \cL^{\otimes t} \rvert$. Then, Koll\'ar--Matsusaka's inequality (\cite[Theorem]{Kollar-Matsusaka}) implies that $I_{d,n}$ is a finite set. Therefore, $\HK_{2n,d}$ is of finite type over $k$. 
    \item The theory of locally trivial deformation given in \cite[Theorem 4.7, Lemma 4.13]{BL22} implies that $\HK_{2n,d}$ is smooth over $k$.
    \item  \Cref{cor:polarizedMatsusakaMumford} implies that the DM stack $\HK_{2n,d}$ is separated over $k$. \endproof
\end{enumerate}
\end{proof}

\subsection{(Weak) polarization and period map}
In this part, we will always assume $k = \CC$. 

Fix a connected component $\HK_{2n,d}^{\circ}$ of $\HK_{2n,d}$ that contains a fixed $\CC$-base point $[X_0,\lambda_0]$, with $h = c_1(\lambda_0) \in \Lambda \coloneqq \rH^2(X_0,\ZZ)$. Let $\Lambda_h \coloneqq h^{\bot} \subseteq \Lambda$ be the sublattice given by the orthogonal complement of $h$, which is of signature $(2, b_2(X)-3)$. The period domain $\Omega_{\Lambda_h}$ of the orthogonal group  $\SO(\Lambda_h)_{\QQ}$ is defined as a connected component of 
\[
\Omega_{\Lambda_h}^{\pm} \coloneqq \Set*{ [\sigma] \in \PP(\Lambda_h \otimes \CC) \given q(\sigma) =0, q(\sigma, \bar{\sigma})>0}.
\]

There is an arithmetic subgroup $\Gamma_h \subseteq \mathrm{O}(h^{\bot})$ (see \cite[Theorem 8.2 (1)]{BL22} and \cite[\S 8]{Markmansurvey}) and a period map 
 \begin{equation}\label{eq:polarziedPeriodMap}
     \cP_{h} \colon \HK_{2n,d}^{\circ,\an} \to [\Gamma_h \backslash \Omega_{\Lambda_h}^{\pm}],
 \end{equation}
associated with the polarized variation of Hodge structure $\underline{L}^{\bot} \subset R^2\pi_* \ZZ$ for the universal polarized family $(\mathcal{X} \xrightarrow{\pi} \HK^{\circ}_{2n,d}, \underline{L})$.
Borel's algebraicity theorem, or more generally, the o-minimal GAGA principle (\cite[Theorem 1.1]{BBT23}) implies that $\cP_h$ is algebraic.

\begin{proposition}\label{prop:quasifiniteperiodmap}
    The period map $\cP_h$ is  quasi-finite. More precisely, for any étale atlas $U$ of a connected component $\HK_{2n,d}^{\circ}$, the period map morphism $\pi \colon U \to [\Gamma_h \backslash \Omega_{\Lambda_h}^{\pm}]$ is quasi-finite to its image. 
\end{proposition}
\begin{proof}
Let $U$ be an étale neighborhood of a point $(X,L)\in \HK_{2n,d}^{\circ}$.
The analytic completion of the period map $\pi \colon U \to [\Gamma_h \backslash \Omega_{\Lambda_h}^{\pm}]$ at the point $(X,L)$ is given by
\[
\Def^{\lt}(X,L) \to \Omega_{\Lambda_h}^{+} \subseteq \Omega_{\Lambda}
\]
which is a local isomorphism by the local Torelli theorem \ref{prop:localTorelliThm} and \cite[Lemma 4.13]{BL22}. This implies that the $\pi \colon U \to [\Gamma_h \backslash \Omega_{\Lambda_h}^{\pm}]$ is (formally) unramified for any étale atlas $U \to \HK_{2n,d}^{\circ}$.
   Hence $\cP_h$ is locally quasi-finite by \cite[Lemma 0H2Z]{stacks-project}. By \Cref{thm:ModuliStack}, the moduli stack $\HK_{2n,d}^{\circ}$ is known to be of finite type over $\CC$. Therefore $\cP_h$ is quasi-finite.
\end{proof}

In general, a locally trivial family is not necessarily a projective morphism.  For this reason, we propose the following weaker notion of polarization and show that it always exists.

\begin{definition}\label{def:weakpolarization}
Let $(B,0)$ be a pointed connected variety.
Let $\pi\colon \mathcal{X} \to B$ be a locally trivial family of primitive symplectic varieties. Let $\Lambda$ be the lattice $\rH^2(\mathcal{X}_0,\ZZ)$.
\begin{enumerate}
    \item A \emph{weak polarization} of $\pi$ is a global section $\underline{L} \in \Gamma(B, (R^2\pi_* \ZZ)_{\tf})$ such that
    \begin{enumerate}
        \item $\underline{L}_b$ is $(1,1)$-class for any point $b \in B(\CC)$ with $q(\underline{L}_b) >0$, and
        \item $\underline{L}_{b_0} = c_1(\cL)$ for an ample line bundle $\cL$ on $\mathcal{X}_{b_0}$ for a very general point $b_0$ of $B$.
    \end{enumerate}
    \item Let $h  \in \Lambda$ with $q(h) >0$. A weak polarization on the pointed family $\cX/B$ is \emph{of type $[h]$} if $c_1(\underline{L}_0) \in [h]= \operatorname{O}(\Lambda) \cdot h$.
\end{enumerate}
\end{definition}

We shall remark the finiteness of polarization types on a pointed family.
\begin{lemma}\label{lem:FiniteH}
    Let $m$ be a positive integer. There are finitely many $\mathrm{O}(\Lambda)$-orbit $[h]$ such that the Beauville--Bogomolov square $q(h) = m$. In particular, for a pointed family of primitive symplectic varieties, a polarization of degree $d$ has finitely many possible equivalent polarization types.
\end{lemma}
\begin{proof}
Any such $h$ determines an embedding of lattices $\langle m \rangle \hookrightarrow \Lambda$, and two $h$ and $h'$ lie in the same $\mathrm{O}(\Lambda)$-orbit if and only if their corresponding embeddings are isomorphic. By \cite[Satz (30.2)]{Kneser02}, there are only finitely many such embeddings up to the action of $\mathrm{O}(\Lambda
)$. 
\end{proof}

We observe that on a locally trivial family, weak polarization always exists even the family is not projective. Moreover, the period map of the weight-two variation of Hodge structure, which is polarized by the weak polarization, has a rational lifting to the moduli stack.
\begin{proposition}\label{prop:RationallyExtended}
    Suppose $\pi\colon\mathcal{X} \to B$ is a locally trivial family of primitive symplectic varieties, with base point $0 \in B(\CC)$. Let $(\Lambda,q) = (\rH^2(\mathcal{X}_0,\ZZ),q_X)$. 
    Then there exist 
    \begin{enumerate}
        \item $d>0$, $h \in \Lambda$ with $q(h) =d$,
        \item a rational map $\phi_{\pi}\colon B\dashrightarrow \HK^{\circ}_{2n,d}$ to a connected component of $\HK_{2n,d}$, and
        \item a weak polarizaton $\underline{L}$ on $\pi$ of type $[h]$
    \end{enumerate}
    such that the composition of $\phi_{\pi}$ with the period map $\cP_{h}\colon \HK^{\circ}_{2n,d}\to [\Gamma_h \backslash \Omega_{\Lambda_h}^{\pm}]$ is a well-defined morphism. In short we have a commutative diagram:
\begin{equation}\label{eq:PointedRationalMaptoModuli}
\begin{tikzcd}
    B \ar[r,dashed, "\phi_{\pi}"] \ar[dr] & \HK^{\circ}_{2n,d} \ar[d,"\cP_{h}"] \\
    &\lbrack \Gamma_h \backslash \Omega_{\Lambda_h}^{\pm} \rbrack.
\end{tikzcd}
\end{equation}
\end{proposition}
\begin{proof}

Let $\eta$ be a generic point of $B$.
Since $\mathcal{X}_{\eta}$ is a projective primitive symplectic variety, there exists an open subscheme $U \subset B$ such that 
$\mathcal{X}_{U}$ is a projective scheme over $U$.
Let $L$ be a $\pi_U$-ample line bundle on $\mathcal{X}_{U}$. In other word, the restricted family $\mathcal{X}_U \to U$ admits a polarization $\underline{L}$. The pair $(\mathcal{X}_U\to U, \underline{L})$ determines a morphism $j \colon U \to \HK_{2n,d}$. Shrinking $U$ if necessary, we may assume that $U$ is connected.
 Thus we get a rational map from $B$ to a connected component of $\HK_{2n,d}$
\[
\phi\colon B \dashrightarrow \HK^{\circ}_{2n,d} 
\]
defined over $U \subseteq B$. By the construction of $\cP_h$, the morphism $\cP_h \circ j \colon U \to \Gamma_h \backslash \Omega_{\Lambda_h}^{\pm}$ is the period map determined by the primitive part $\mathbf{P}^2(\cX_U/U)$.

Recall that $R^2 \pi^{\an}_{\ast} \ZZ$ is a local system on $B^{\an}$ by \Cref{prop:PropLocallyTrivial}.
Since $U \subset B$ is a Zariski open subset, the natural morphism
\[
\pi_1 (U^{\an}) \rightarrow \pi_1 (B^{\an})
\]
is surjective.
Therefore, the section $c_1 (L^{\an}) \subset \Gamma (U^{\an}, R^2 \pi^{\an}_{\ast} \ZZ)$ extends to a section $\underline{L} \in \Gamma (B, R^2 \pi^{\an}_{\ast} \ZZ)$. Since $\underline{L}_{b_0}$ is a $(1,1)$-class for any point $b_0 \in U^{\an}$ and $U^{\an} \subset X^{\an}$ is dense, $\underline{L}_b$ are all Hodge $(1,1)$ classes for all points $b \in B^{\an}$. In particular, $\underline{L}$ is a weak polarization on $\cX/B$. The  type of weak polarization $\underline{L}$ with respect to the base point $b$ is $h = c_1(\underline{L}_b) \in \Lambda$.

The orthogonal complement $\underline{L}^{\bot} \subseteq R^2\pi_*\ZZ_{\tf}$ is polarized by the restriction of Beauville--Bogomolov form since it satisfies the Hodge--Riemann relations as in the polarized case. Therefore, $\cP_h \circ j$ extends to the (algebraic) period map
 \begin{equation}\label{eq:PeriodMapWeakPolarization}
     \cP_{\underline{L}^{\bot}} \colon B \to \lbrack\Gamma_h \backslash \Omega_{\Lambda_h}^{\pm}\rbrack.
 \end{equation}
 associated with $\underline{L}^{\bot}$. The commutativity of the diagram is clear. \qedhere
 
\end{proof}

\begin{remark}We note that the pointed Shafarevich set \eqref{eq:pointshaf} is countable up to isomorphism. {Assume that $B=C$} is a connected smooth curve. Applying Zariski's main theorem for the period map $\cP_h$, we see there is a closed subset $\Delta_d \subset C$ such that for locally trivial family $\pi$ in $\eqref{eq:pointshaf}$ that admits a weak polarization of degree $d$, the rational map $\phi_{\pi}$ has definition outside $C \setminus \Delta_d$. As $C$ is a curve, $\Delta_d$ is a finite set, and for any $b \in \Delta_d$, the fibers are finite up to birational equivalence. Now, it is not hard to see the countability from \Cref{lem:finitePolarizedAut} and the Matsusaka--Mumford theorem.

\end{remark}

\subsection{Proof of \Cref{thm:polarizedfiniteness}}\label{sec:FinitePolarizedHK}

Let $\overline{\eta}$ be a geometric generic point of $B$.
Note that we have a natural map 
\begin{equation}\label{eq:Geometricalization}
\Shaf_{B,X,d} \to \overline{\Shaf}_{B,X,d},\quad
    (\cX_{\eta}, L) \mapsto (\cX_{\bar{\eta}}, L_{\overline{\eta}} )
\end{equation}
where
\[
    \overline{\mathrm{Shaf}}_{B,X,d} \coloneqq
    \{
    (\mathcal{X}_{\overline{\eta}}, L_{\overline{\eta}}) \mid (\mathcal{X}_{\eta}, L) \in \mathrm{Shaf}_{B,X,d}
    \}/\simeq_{\overline{\eta}}.
    \]
It is sufficient to show that $\overline{\Shaf}_{B,X,d}$ is finite and that the fibers of \eqref{eq:Geometricalization} are finite.

\begin{step}[1]
In the setting of \Cref{thm:polarizedfiniteness}, the set
    \(\overline{\mathrm{Shaf}}_{B,X,d}\)
    is a finite set.
\end{step}

As before, let $\Lambda$ denote $\rH^2(X_{\CC},\ZZ)_{\tf}$. Let $\cX \to B$ be a locally trivial family with $\cX_0 \cong X$ (over $k$). For any $(\cX_\eta,L) \in \Shaf_{B,X,d}$ with $L$ an ample line bundle, it defines a point $x \in \HK_{2n,d}^{\circ} (k(\eta))$ for some connected component $\HK^{\circ}_{2n,d} \subseteq \HK_{2n,d}$.

We may assume $k \subseteq \CC$.
By \Cref{prop:RationallyExtended}, there exist $h$ such that we have a commutative diagram
\[
\begin{tikzcd}
    B_{\CC} \ar[r,dashed, "x"] \ar[dr,"\widetilde{x}"] & \HK^{\circ}_{2n,d} \ar[d,"\cP_{h}"] \\
    & \Gamma_h \backslash \Omega_{\Lambda_h}^{\pm}
\end{tikzcd}
\]
such that  $\widetilde{x}(\eta) = \cP_{h,F}(x)$, and
    $\widetilde{x}(0_{\CC}) $ is the $\Gamma_h$-equivalent classes of polarized $\ZZ$-Hodge structures $[H^2_{\prim}(X,\ZZ)]$ for some $h \in \Lambda$. \Cref{lem:FiniteH} implies that, up to $O(\Lambda)$, there are only finitely many possible choices of $h$ with a fixed $q_X (h)$. 
    Therefore, there are finitely many targets $\Gamma_h \backslash \Omega_{\Lambda_h}^{\pm}$ for $\widetilde{x}$ as $x$ varies in $\Shaf_{B,X,d}$,  because $q_X(h)^{ \dim X/2} = d/c_X$ for the Fujiki constant $c_X$ of $X$ ({ cf.\ \cite[Subsection 5.14]{BL22}}). Therefore, we may assume that all $\widetilde{x}$ have the same target $\Gamma_h \backslash \Omega_{\Lambda_h}^{\pm}$ for some $h \in \Lambda$.

In this case, we shall note that $\widetilde{x}(0_{\CC})$ are also all the same. Therefore, by \cite[Theorem 6.1, Lemmas 2.4--2.6]{JL23},
the number of isomorphism classes of $\widetilde{x}$ (where $x$ varies in $\mathrm{Shaf}_{B,X,d}$) is finite, which implies that the number of isomorphism classes of $\cP_{h} (x)$ is finite.
Since $\cP_{h}$ is quasi-finite, $\cP_{h}$ on the groupoid over $\overline{\eta}$ is finite-to-one modulo isomorphisms.
That means, the isomorphism classes of $\widetilde{x}(\overline{\eta})$ are finitely many, and it finishes the proof.

\begin{step}[2]
 Finiteness of twists with locally trivial reduction over $B$. 
\end{step}
More precisely,
 Let $(B,0)$ be a pointed smooth variety with the generic point $\eta$.
Let $\mathcal{X}$ be a locally trivial family of primitive symplectic varieties over $B$, and $L$ be a polarization on $\mathcal{X}_{\eta}$.
Then the set
\[
\mathrm{Tw}:=
\left\{ 
(\mathcal{Y}_{\eta}, M)
\left| 
\begin{array}{l}
\mathcal{Y} \textup{ is a locally trivial family } \\
\textup{of primitive symplectic varieties over $B$},\\
M \textup{ is a polarization on $\mathcal{Y}_{\eta}$}, \\
\textup{there exists $(\mathcal{Y}_{\overline{\eta}}, M_{\overline{\eta}}) \simeq_{\overline{\eta}} (\mathcal{X}_{\overline{\eta}}, L_{\overline{\eta}})$}
\end{array}
\right.\right\}
/ \simeq_{\eta}
\]
is a finite set.

This follows from the same proof as in \cite[Proposition 6.3]{FLTZ22} by using the Matsusaka--Mumford theorem (Proposition \ref{prop:birationalspecialization}) and the Hermite--Minkowski type theorem (\Cref{HermiteMinkowski}).
We sketch the argument in the following.
Let $ G:= \underline{\Aut}(\mathcal{X}_{\eta},L)$ be the automorphism group scheme, which is a finite group scheme by \Cref{lem:finitePolarizedAut}.
We may take a finite Galois extension $\eta'/\eta$ such that $G(\overline{\eta}) = G (\eta')$.
By the finiteness of 
\begin{equation}
\label{eqn:firstGalois}
H^1 (\Gal (\eta'/\eta),  \underline{\Aut}(\mathcal{Y}_{\eta'},M_{\eta'})),
\end{equation}
for any $(\mathcal{Y}_{\eta}, M) \in \mathrm{Tw}$, we may replace $\eta$ by $\eta'$, i.e.\ we may assume that $G(\overline{\eta}) = G(\eta)$.
Then each $(\mathcal{Y}_{\eta}, M) \in \mathrm{Tw}$ 
(more precisely, the isomorphism $f\colon (\mathcal{Y}_{\overline{\eta}}, M_{\overline{\eta}}) \simeq (\mathcal{X}_{\overline{\eta}}, L_{\overline{\eta}})$)
defines a 1-cocycle
\[
\alpha_{f} \colon
\Gal (\overline{\eta}/\eta) \rightarrow G(\overline{\eta}) = G(\eta),
\]
which is a group homomorphism.
For a codimension 1 point $b \in B$, let $I_{b} \subset \Gal (\overline{\eta}/\eta)$ be the inertia subgroup, which is defined after fixing the extension of valuation corresponding to $b$ to $\overline{F}$.
By the construction of $\alpha_{f}$, we have
\[
\overline{\psi} \circ \alpha_{f} (\sigma) = 1
\]
for $\sigma \in I_b$,
where 
\[
\overline{\psi} \colon G (\eta) \rightarrow  \Aut ((\mathcal{X}\setminus V)_{b})
\]
is the specialization morphism defined in \Cref{prop:birationalspecialization} with respect to $\cO_{B,b}$.
Since $\overline{\psi}$ is injective by \Cref{prop:birationalspecialization}, we have $ I_b \subset \ker \alpha_{f}$ for any $b$.
This shows that $\alpha_f$ factors through $\pi_1^{\et} (\Spec B, \overline{\eta})$ by Zariski--Nagata's purity.
Since $\# \ker \alpha_f \leq \# G(\eta)$, by \Cref{HermiteMinkowski}, there exists a finite Galois extension $\eta'/\eta$ that is independent of $(\mathcal{Y}_{\eta}, M)$ such that
$\alpha_{f} |_{\Gal (\overline{\eta}/\eta')}$ is trivial, i.e.\ $(\mathcal{Y}_{\eta'}, M_{\eta'}) \simeq (\mathcal{X}_{\eta'}, L_{\eta'})$.
By the finiteness of (\ref{eqn:firstGalois}) again,  it finishes the proof. \qed

\section{Finiteness of the generic fibers}\label{sec:Finiteness of the generic fibers}

As before, the base  \((B, 0)\) is a pointed smooth variety over an algebraically closed field \(k \) of characteristic zero, with generic point \(\eta\). Let \(X\) be a projective primitive symplectic variety over \(k\).
 The goal of this section is to prove the following refinement of Theorem \ref{thm:introgenfinite}:

\begin{theorem}\label{thm:FiniteGenericFibers}
Suppose that $X$ is $\QQ$-factorial terminal.  If \( b_2(X) \neq 4 \),  the set 
\begin{equation}\label{eq:genfiber}
\left\{ 
\mathcal{X}_{\eta}
\left| \begin{array}{l}
\mathcal{X} \textup{ is a locally trivial family} \\
\textup{of primitive symplectic varieties over $B$ }\\
\textup{such that $\mathcal{X}_0\simeq X$}\\

\end{array}
\right.\right\}
/ \simeq_{\eta},
\end{equation}
is finite, where \(\simeq_{\eta}\) denotes isomorphism of generic fibers.   If \( b_2(X) = 4 \), the same finiteness holds for non-isotrivial families in \eqref{eq:genfiber}.
\end{theorem}
Here, a locally trivial family $\cX\to B$ is called \textit{isotrivial} if there exists an étale surjective morphism \(B' \to B\), such that \(\mathcal{X}_{B'} \simeq X \times B'\) as \(B'\)-schemes.

\subsection{Construction of Uniform Kuga--Satake map}  
In the proof, we use the so‐called “uniform Kuga–Satake construction,” as treated in \cite{She}, \cite{OS18} and \cite{FLTZ22} as a variant of Zarhin’s trick for primitive symplectic varieties, to reduce the problem to the finiteness of families of polarized primitive symplectic varieties, which was established in \Cref{sec:FinitePolarizedHK}

{In this subsection, we work with the base field $ k= \CC$.} For any polarized weight-$2$ Hodge structure $\rV$ such that $\rV^{2,0}=1$ and $\rV\otimes \RR$ is of signature $(2,m)$, one can associate it with a polarized  abelian variety $(A_\rV,\Phi_a)$ of dimension $2^{m+1}$  with 
$$\rH^1(A_\rV,\ZZ)\cong \Cl(\rV)$$
where $\Cl(\rV) = \Cl^+(\rV) \oplus \Cl^-(V)$ is the Clifford algebra  of  $\rV$ with the $\ZZ/2$-grading. The polarization $$\Phi_a(x,y):=\mathrm{Trace}(\iota(x)ya)$$ on $\rH^1(A_{\rV},\ZZ)$  depends only on the lattice $\rV$ and an element $a \in \Cl^+(\rV)$ such that $\iota(a) = - a$, where $\iota$ is the involution on $\Cl(\rV)$ (cf.~\cite[\S 5.4]{Rizov10}). Its degree $d=d(a,\rV)$ can be computed explicitly in terms of $a$ and $\rV$. Moreover, there is a natural inclusion of sub-Hodge structures
\[
\rV(1) \hookrightarrow \End(\rH^1(A_\rV, \ZZ), \rH^1(A_\rV, \ZZ)).
\]
The abelian variety $A_\rV$ is called the \emph{full Kuga--Sataka variety} of $\rV$.

Let $\pi\colon \mathcal{X} \to B$ be a locally trivial family of primitive symplectic varieties, with a weak polarization $\underline{L}$. Let 
\[
\mathbf{P}^2(\cX/B) \coloneqq \underline{L}^{\bot} \subseteq R^2\pi_* \ZZ_{\tf}
\]
be the associated variation of Hodge structure of K3 type, where $(-)^{\bot}$ is the orthogonal complement with respect to the Beauville--Bogomolov form.

Let $\bA_{d,g,n}$ be the moduli stack of abelian varieties of dimension $g$, with polarization of degree $d^2$ and a level-$n$ structure. 
The  (relative) Kuga--Satake construction for $\mathbf{P}^2({ \pi})$ induces a map
\begin{equation}\label{eq:pks}
    B \to \bA_{d,g,n}
\end{equation}
after a finite \'{e}tale extension of $B$, where $g = 2^{b_2-2}$ and $d,n$ are some positive integers. At each $\CC$-point $b$ of $B$, the image is the Kuga--Satake variety $A_{L_b^{\bot}}$.
We point out that the polarized Kuga--Satake map \eqref{eq:pks}, and hence the degree $d$, depends on the polarization type  $h$ of $\underline{L}$. The following theorem is the main result of this section, the key point being that the \textit{uniform} Kuga--Satake map is independent of the family $\pi \colon \cX\to B$.

\begin{theorem}[Uniform Kuga--Satake]
\label{thm:uniformKugaSatake}
Let $X$ be a primitive symplectic variety. Let $(B,0)$ a pointed connected complex variety. There exist integers $d,g,n > 0$ and a finite étale covering $u: \widetilde{B} \to B$ such that for any locally trivial family of primitive symplectic varieties $\pi\colon \mathcal{X} \to B$ equipped with a weak polarization $\underline{L}$ and $\pi^{-1}(0) \cong X$, there is a morphism 
\[\KS_{\pi,\underline{L}}\colon \widetilde{B} \to \mathbf{A}_{d,g,n},\]
 $\ZZ$-variation of Hodge structure $\underline{\rH^{\sharp}}$ on $\widetilde{B}$ of weight zero that is independent of $\pi$, and an embedding of variations of Hodge structure  \begin{equation}\label{eq:HodgeembeddtoKS}
u^*\mathbf{P}^2(\cX/B)(1) \hookrightarrow \underline{\rH^{\sharp}} \hookrightarrow \underline{\End}(R^1p_{*}\mathbb{Z}),\end{equation}
where $p: \mathcal{A}_{\pi,\underline{L}} \to \widetilde{B}$ is the pulled-back via $\KS_{\pi,\underline{L}}$ of the universal  abelian scheme over $\mathbf{A}_{d,g,n}$.
\end{theorem}

\begin{proof}
Let $\Lambda$ be the lattice determined by $\rH^2(X,\ZZ)_{\tf}$ with the Beauville--Bogomolov form. Let $ \underline{L}$ be a weak polarization on $\mathcal{X} \to B$ of type $h$.

Consider the subgroup of $\SO(\Lambda)(\ZZ)$ defined as
\[
G_h \coloneqq \Set{g \in \SO(\Lambda)(\ZZ)\given g h = h}.
\]
Fix a connected component $\Omega_{\Lambda_h} \subseteq \Omega_{\Lambda_h}^{\pm}$. The group $G_h$ naturally acts on $\Omega_{\Lambda_h}$. By passing to a degree two covering $B'$ of $B$, we may assume the period map $\cP_{\underline{L}^{\bot}}$ in \eqref{eq:PeriodMapWeakPolarization} lifted to a morphism
\[
B' \to [G_h \backslash \Omega_{\Lambda_{h}}].
\]

There is a unimodular even lattice $\Uni$ of signature $(2,m)$ such that for any $h$ with $q(h) >0$, lattice $\Lambda_h$ or $\Lambda_h(2)$ admits a primitive embedding  into $\Uni$ by \cite[Theorem 1.12.4]{Nikulin}. The Kuga--Satake construction gives a morphism of Deligne--Mumford stacks \[\KS \colon [\Gamma_{\Uni} \backslash \Omega_{\Uni}] \to \bA_{g,d,n},\] where $\Gamma_{\Uni} = \SO(\Uni) \cong \SO^+(\Uni)$ is the special orthogonal group of lattice $\Uni$ (since $\Uni$ is unimodular). There is an integer $N$ and a finite morphism of Deligne--Mumford stacks $j \colon [\Gamma_h^{\spin}(N)\backslash \Omega_{\Lambda_h}]\to [\Gamma_{\Uni}\backslash \Omega_{\Uni}]$,
where $\Gamma^{\spin}_h(N)$ is the arithmetic subgroup of $G_h$ given by a spin level-$N$ structure. The composition of $\KS$ and $j$ is denoted by 
\[
\beta_{h,N} = \KS \circ j \colon [\Gamma^{\spin}_h(N) \backslash \Omega_{\Lambda_h}] \to \bA_{g,d,n}.
\]
By \cite[Proposition 3.8(a)]{Pink}, we can see $\KS$ and $j$ are finite morphisms. Hence the same holds for $\beta_{h,N}$.

Consider the finite étale morphism of complex analytic stacks $[\Gamma_h^{\spin}(N) \backslash \Omega_{\Lambda_h}] \to [G_h\backslash \Omega_{\Lambda_h}]$ whose degree is equal to the index $m = [G_h : \Gamma_h^{\spin}(N)]$. By \cite[Lemma 7.1]{FLTZ22}, we can see
\[
m \leq C N^{2^{b_2(X) -2}},
\]
where $C$ is a constant depending only on $X$.
Taking the following Cartesian diagram
\[
\begin{tikzcd}
    \widetilde{B}_{\pi} \ar[r] \ar[d, "p_h"] & \Gamma_h^{\spin}(N) \backslash \Omega_{\Lambda_h}  \ar[d] \ar[r,"\beta_{h,N}"] & \bA_{g,d,n}\\
    B' \ar[r] & \lbrack G_{h} \backslash \Omega_{\Lambda_h}\rbrack.
\end{tikzcd}\]
The morphism $p_h$ is finite étale of degree $m$. 
{The following \Cref{HermiteMinkowski}, a topological version of the Hermite--Minkowski theorem, implies that there are only finitely many possible coverings $p_h$ when the family $\pi \colon \cX \to B$ varies. Thus we can take $\widetilde{B} \to B$ to be the (finite) fiber products of all such finite coverings $\widetilde{B}_\pi \to B$, which is equipped with a morphism $\widetilde{B} \to \Gamma_h^{\spin}(N) \backslash \Omega_{\Lambda_h}$ and is independent of the family $\pi$ as required.} 
Its composition with $\beta_{h,N}$ gives the required uniform Kuga--Satake map, which is quasi-finite since $\widetilde{B} \to \Gamma_h^{\spin}\backslash \Omega_{\Lambda_h}$ is quasi-finite.
\endproof

\end{proof}

\begin{lemma}\label{HermiteMinkowski}
Let $d$ be a positive integer. Suppose $B$ is a connected variety over $\CC$.
There are at most finitely many finite étale surjective morphisms $p \colon B' \to B$ with degree $\leq d$.
\end{lemma}
\begin{proof}
Such a finite étale morphism of degree $d$ corresponds to a group homomorphism
\[
\pi_1^{\et}(B) \to \mathfrak{S}_d
\]
to the symmetric group of degree $d$.
   Since the topological fundamental group $\pi_1(B^{\an})$ is finitely presented by \cite{Hironaka} and \cite{Lojasiewicz}, its profinite completion $\pi_1^{\et}(B)$ is topologically finitely presented. Therefore, there are only finitely many such group homomorphisms.
\end{proof}

Let $X$ be a primitive symplectic variety over $\CC$ and $L$ a line bundle on $X$ with $q(L) >0$. Suppose that there is a locally trivial family $\pi \colon \cX \to B$ of primitive symplectic varieties such that $\pi^{-1}(0) \cong X$ for a closed point $0 \in B$ and $c_1(L) = \underline{L}_0$ for a weak polarization $\underline{L}$ over $\pi \colon \cX \to B$. 
Let $\widetilde{B} \rightarrow B$ be as in Theorem \ref{thm:uniformKugaSatake}, and we fix $\widetilde{0} \in \widetilde{B} (\CC)$ that is a lift of $0 \in B(\CC)$.
Denote \begin{equation}\label{eq:KS_notation}
    \KS_{\pi,\underline{L}}(X) \in \bA_{d,g,n}(\CC)
\end{equation}
for $\KS_{\pi,\underline{L}}|_{\widetilde{0}}$.
\begin{proposition}\label{prop:KS-finiteness}
  Let \( X \) be a primitive symplectic variety over $\CC$. The following set of uniform Kuga--Satake varieties is finite:
  \[
    \Set*{\KS_{\pi,\underline{L}}(X) \given \parbox{50ex}{$\pi\colon \mathcal{X}\to B$ is a locally trivial family of primitive symplectic varieties with weak polarization \( \underline{L} \) such that $\cX_0 \coloneqq \pi^{-1}(0) \cong X$} }.
  \]
\end{proposition}
\begin{proof}
The key observation is that for any weak polarization \( \underline{L} \), the associated polarized abelian variety \( \KS_{\pi,\underline{L}}(X)\in \bA_{d,g,n} \) is completely determined by the induced polarized integral Hodge structure on $\rH^{\sharp}$ together with a level structure, and the integral Hodge structure is given by the following datum: 
  \begin{itemize}[leftmargin=*]
    \item The transcendental Hodge structure \( \rT(X) \), which is determined by  \( X \) up to isometry;
    \item A primitive lattice embedding \( \iota_L: \rT(X) \hookrightarrow c_1(L)^\perp \hookrightarrow \Uni \), where \( \Uni \) denotes the universal Kuga--Satake lattice given in \Cref{thm:uniformKugaSatake}.
  \end{itemize}
Since the period satisfies 
\[
(\rH^{\sharp})^{2,0} = \iota_L\left(\rT(X)^{2,0} \right)
\]in $\rH^{\sharp}\otimes \CC$ by \eqref{eq:HodgeembeddtoKS}, it is sufficient to see the finiteness of such primitive embeddings $\iota_L$. For this, we can use Nikulin's theorem \cite[Theorem 3.6.3]{Nikulin}, which asserts that for a fixed even integral lattice \( \Uni \) and a sublattice \( \rT(X) \) of signature \( (2, n) \), there exist only finitely many primitive embeddings 
  \(\rT(X) \hookrightarrow \Uni\) up to the action of \( \rO(\Uni) \). 
\end{proof}

\subsection{Finiteness of geometric N\'eron--Severi lattices }

We first establish the finiteness of the Néron-Severi lattice of closed fibers on pointed families.

\begin{proposition}\label{prop:FiniteTran} Let $(B,0)$ be a connected smooth variety over $\CC$. Let $X$ be a primitive symplectic variety.
Let \(b \in B\) be a fixed closed point. Then the set of \emph{Néron–Severi lattices at $b$}
    \[
    \NS_{b} \coloneqq \left\{ \NS(\mathcal{X}_{b}) \ \Bigg| ~\cX\to B~\text{is a pointed family as in \eqref{eq:pointshaf}} \right\}{\big/ \text{isometry}}
    \]
is finite. 
\end{proposition}  

\begin{proof} 

Let \(\pi: \widetilde{B} \to B\) be the finite étale covering given in Theorem~\ref{thm:uniformKugaSatake} and fix a point \(\widetilde{b} \in \widetilde{B}\) (resp.\ $\widetilde{0} \in \widetilde{B}$) over \(b\) (resp.\ $0$).  Then we obtain a morphism 
\[\varphi: \widetilde{B}\to \bA_{d,g,n}\]
for any $\mathcal{X}\rightarrow B \in \eqref{eq:pointshaf}.$
Thanks to \Cref{prop:KS-finiteness}, though the map $\varphi$ depends on $\cX\to B$, all possible $\varphi(\widetilde{0})$ form a finite set. Therefore, we can fix $\varphi(\widetilde{0})$. The geometric hyperbolicity of \(\bA_{g,d,n}\) (see \cite{JL23}) implies that  there are only finitely many possibilities for $\varphi$. 

Thus for any fixed $b\in B$, the associated uniform Kuga--Satake variety of $\cX_b$  only has only finitely many possibilities.  
Then the  same argument in \cite[Theorem 7.4]{FLTZ22} (using the Lefschetz (1,1) theorem instead of the Tate conjecture) shows that $\NS(\cX_b)$ has finitely many possibilities. 
\end{proof}

As an application, we get the following consequence.
\begin{corollary}
\label{cor:nsbound}
Let $(B,0)$ be a pointed smooth variety over $k$  with the generic point $\eta$.
Let $X$ be a primitive symplectic variety over $k$.
Then
\begin{equation}\label{eq:NSShafSet}
\NS_{\Shaf}\coloneqq
\left\{ 
\NS(\mathcal{X}_{\overline{\eta}})
\left| \begin{array}{l}
\mathcal{X} \textup{ is a locally trivial family} \\
\textup{of primitive symplectic varieties over $B$}\\
\textup{such that $\mathcal{X}_0\simeq X$}\\
\end{array}
\right.\right\}
/ \textup{isometry}.
\end{equation}
is a finite set.
\end{corollary}
\begin{proof}
\Cref{prop:RationallyExtended} and its proof imply that any such pointed locally trivial family $\cX/B$ admits polarization on its generic fiber. Therefore, the set \eqref{eq:genfiber} of generic fibers is a countable set by \Cref{thm:polarizedfiniteness}. Therefore, there exists a countable set of locally trivial families
    \[
    I \coloneqq \Set*{ f_i  \colon \cX_i \to B}_{i \in \NN}
    \]
    such that the set \eqref{eq:genfiber} is equal to  $\Set{\cX_{i,\eta} \given \cX_i \in I}$. 
    By the spreading-out argument, we may assume that the elements in $I$ are defined over a countable extension of $\QQ$. Therefore, we may assume that $k =\CC$.
    \Cref{lem:generalpointPicard} implies that there is a point $b_0 \in B (\CC)$ such that for any $Y$ in \eqref{eq:genfiber}, there is a locally trivial family of primitive symplectic varieties $f \colon \cX \to B$ with  $\cX_{\eta} \cong Y$, and  $\NS(\cX_{b_0}) \cong \NS(Y_{\overline{{\eta}}})$. Thus there is a point $b_0\in B(\CC)$ such that
\[
\NS_{\Shaf} = \NS_{b_0}.
\]
Now the statement follows from \Cref{prop:FiniteTran}.
\end{proof}

\subsection{Proof of \Cref{thm:FiniteGenericFibers}}

We split the proof into three parts according to the second Betti number.

(1).~
If $b_2(X) = 3$, then any locally trivial algebraic deformation of $X$ is trivial, so the geometric fibers of $\cX \to B$ are all isomorphic to $X$.  %
In this case, we have $\mathcal{X}_{\overline{\eta}} \simeq X \times_k \overline{\eta}$ and $\rho (\mathcal{X}_{\overline{\eta}}) =1.$
Therefore, one can use the same argument as in Step (2) of Proof of Theorem \ref{thm:polarizedfiniteness} to obtain the finiteness.

(2).~ {If $b_2(X) =4$, note that the statement only concerns non-isotrivial families. We claim that if  $\cX\to B$ is not isotrivial, then the very  general fiber of $\cX\to B$ has Picard number one. Indeed, if the geometric generic fiber has Picard number 2, then after shrinking $B$, we may assume that $\cX \to B$ is projective and admits a rank two lattice-polarization in the sense of \cite{LLX24}.  Since the deformation space of locally trivial rank two lattice-polarized primitive symplectic varieties is trivial, this forces  $\cX\to B$ to be isotrivial (see also \cite{BKPS98}, \cite{Oguiso} for a different approach), contradicting to the hypothesis.  Therefore, we proved that for any non-isotrivial family $\cX\to B$,  its geometric generic fiber $\cX_{\overline{\eta}}$ has Picard number one.  Then, by \Cref{cor:nsbound}, the minimal polarization degree of $\mathcal{X}_{\eta}$ is bounded. %
We obtain the desired result by \Cref{thm:polarizedfiniteness}. }

(3). Suppose that $b_2(X) \geq 5$.
Given an element $\mathcal{X}_{\eta} \in \Shaf_{B,X}$, let $\rho = \rho(\cX_{\bar{\eta}})$ be the geometric Picard number of $\cX_{\eta}$ and $m = b_2(\cX_{\eta})$ be its second Betti number. For simplicity, we assume that $\Pic(\cX_{\bar{\eta}})$ is torsion-free. Otherwise, we can replace it by the torsion-free part $\Pic(\cX_{\bar{\eta}})_{\tf}$. Note that the image of the representation
\[ 
r\colon \Gal (\overline{\eta}/\eta) \rightarrow \GL (\Pic (\mathcal{X}_{\overline{\eta}}))
\]
is finite and its order is bounded by $3^{\rho^2} \leq 3^{m^2}$, since we have the following short exact sequence (by considering coefficient-wise modulo 3 map for integral matrices):
\[
0 \rightarrow 1 + 3\mathrm{Mat}_{\rho\times \rho} (\ZZ) \rightarrow \GL (\Pic(X_{\overline{\eta}})) \rightarrow \GL_{\rho} (\ZZ/3\ZZ) \rightarrow 0,
\]
and the subgroup $1 + 3\mathrm{Mat}_{\rho\times \rho} (\ZZ) $ is torsion-free. Note that the first Chern class map induces an $\Gal(k(\bar{\eta})/k(\eta))$-equivariant injection
\[
\Pic(\cX_{\bar{\eta}}) \hookrightarrow \rH^2_{\et}(X_{\bar{\eta}},\widehat{\ZZ}) \coloneqq \prod_{\ell \text{ prime}} \rH^2_{\et}(\cX_{\bar{\eta}},\ZZ_{\ell}).
\]
We claim that the higher direct image $R^2\pi_*^{\et}\ZZ/n$ is locally constant for any integer $n \geq 1$. This property is stable
under algebraically closed extensions; thus, it is sufficient to assume $k=\CC$.
\Cref{prop:PropLocallyTrivial} (2) implies that $R^2\pi_*^{\an} \ZZ/n$ is finite locally constant for any integer $n$. Therefore, by \cite[Exposé XI, Théorème 4.4.]{SGA4}, the étale sheaf $R^2\pi_*^{\et}\ZZ/n$ is also finite locally constant since $B$ is smooth\footnote{In general, $R^i\pi^{\et}_*\ZZ/n$ are locally constant without the smoothness of $B$. See \Cref{lem:localconstantlocalTrivial}.}.
Thus, the representation $r$ factors through a morphism \[
\pi_1^{\et}(B) \to \GL(\Pic(\cX_{\bar{\eta}}))\]
along the natural surjection $\Gal(k(\bar{\eta})/k(\eta)) \twoheadrightarrow \pi_1^{\et}(B)$.
Then, by the Hermite--Minkowski type theorem (see \Cref{HermiteMinkowski}), we may take a finite étale morphism $B' \rightarrow B$ from a complex smooth variety such that for any $\mathcal{X}_{\eta} $ in \eqref{eq:genfiber},
\[
\rL \coloneq\Pic_{\mathcal{X}_{\eta'}/\eta'} (\eta') = \Pic (\mathcal{X}_{\overline{\eta}}).
\]
Here, $\eta'$ is the generic point of $B'$. Therefore, in this case, the set \eqref{eq:genfiber} is bijective to the following set
\begin{equation}\label{eq:Union with Marked Geometric Picard}
    \bigcup_{\rL \in \NS_{\Shaf}}\Set*{\cX_{\eta} \given \parbox{43ex}{$\cX$ is a locally trivial family of primitive symplectic varieties over $B$ such that $\cX_0 \cong X$ and $\Pic_{\cX_{\eta'}/\eta'}(\eta') \cong \rL$} } 
\end{equation}
By the argument as in \cite[Lemma 4.3.1]{takamatsu2022}, the lower bound of the squares of MBM classes (\Cref{prop:boundedMBM}) indicates that there exists a positive integer $N(\rL)$ such that, for any $\cX_{\eta}$ in \eqref{eq:Union with Marked Geometric Picard} with $\Pic_{\mathcal{X}_{\eta'}/\eta'}(\eta') \cong \rL$, the base extension $\cX_{\eta'} = \cX_{\eta} \times_{\eta} \eta' $ admits a polarization of degree $ \leq N(\rL)$. The integer $N(\rL)$ only depends on the lattice $\rL$ (and the deformation type of $X$). \Cref{thm:polarizedfiniteness}  and \Cref{prop:finitenessoftwists} imply that the set \eqref{eq:Union with Marked Geometric Picard}  is a union of finite sets. Finally, we can see that the set \eqref{eq:Union with Marked Geometric Picard} is finite since $\NS_{\Shaf}$ is finite by \Cref{cor:nsbound}. \qed

\section{Finiteness of projective models and counter-examples}
Let $X$ be a primitive symplectic variety and $C$ a smooth connected curve. According to \Cref{thm:FiniteGenericFibers}, $X$ has only finitely many birational models over $C$. In this section, we utilize finiteness results from cone conjectures in \Cref{subsec:RelativeCone} to establish the finiteness of isomorphism classes for these models. Furthermore, we illustrate with an example that the assumption of projectivity for families is essential for this inquiry.
\subsection{Proof of \Cref{thm:introisomfinite}}
Let $k $ be an algebraically closed field of characteristic zero.
For simplicity, we assume $(C,0)$ is a smooth integral pointed curve over $k$ and $X$ a $\QQ$-factorial primitive symplectic variety with terminal singularities.  By \Cref{thm:FiniteGenericFibers}, we can see there are only finitely many primitive symplectic varieties $Y$ over the function field $K = k(C)$, such that there is a projective locally trivial family $\cX/C \in \Shaf(C,X)$ with $Y \cong \cX_{K}$. 

Let $\cX'/C \in \Shaf(C,X)$ be another projective locally trivial family such that $\cX'_K \cong Y$ over $K$. 
Since $C$ is normal, there is a Zariski covering
\[
C = \bigcup_{s\in C \text{ closed point}} C_s
\]
that $C_s$ is the localization of $C$ at $s$.
For any closed point $s \in C$, \Cref{prop:unpolarizedMatsusaka-Mumford} implies that it can be uniquely extended to a birational equivalence $\cX_{C_s} \dashrightarrow  \cX'_{C_s}$ over $C_s$ such that the restriction $\cX_s \dashrightarrow \cX_s' $ is a birational equivalence. We can glue it to a birational equivalence $ g\colon \cX \dashrightarrow \cX'$ over $C$. By the construction, we can see $g$ is an isomorphism in codimension one. By \Cref{lemma:Total Space Qfactorial},  the total space $\cX$ of any $\cX/C \in \Shaf(C,X)$ is $\QQ$-factorial, and $g$ is a small $\QQ$-factorial modification over $C$ by the previous discussion. Now, we can apply \Cref{cor:finiteSQM} to conclude it. \qed
\begin{remark}
    \Cref{thm:introisomfinite} holds for pointed regular base variety $(B,0)$ by localizing at codimension-one points.
\end{remark}
\subsection{Proof of \Cref{cor:introFiniteKnonwTypes}}\label{sec:proofofKnownTypes}
Note that the condition in \Cref{thm:introisomfinite} follows from the SYZ conjecture, which says that, on a primitive symplectic variety, the linear system of any isotropic nef line bundle induces a Lagrangian fibration, and in particular, the nef line bundle is semiample.
The SYZ conjecture has been confirmed for smooth irreducible symplectic varieties of all known deformation types:
\begin{itemize}
    \item For $\operatorname{K3}^{[n]}$-type, see \cite[Theorem 1.5]{BayerMacri}; 
    \item for $\operatorname{Kum}^{n}$-type, see \cite[Proposition 3.38]{Yoshioka16};
    \item for OG6-type, see \cite[Corollary 1.3]{MongardiRapagnetta}; 
    \item for OG10-type, see \cite[Theorem 2.2]{MongardiOnorati}.
\end{itemize}
Therefore, \Cref{thm:introisomfinite} implies that finiteness of the pointed Shafarevich set \eqref{eq:pointshaf} for these four known deformation types. \qed
\subsection{A general construction of counter-examples}\label{subsec:CounterExample}
This last section is devoted to proving the following result.
\begin{proposition}\label{prop:eg}
    There exist infinitely many  families of  smooth irreducible symplectic varieties $\pi_i:\mathcal{X}_i\to C$ over some smooth integral complex pointed curve $(C, 0)$ satisfying that 
\begin{enumerate}
    \item $\cX_i$ are isomorphic over  $ C \setminus \{ 0\}$ for all $i,j$;
    \item the special fibers $\pi_i^{-1}(0)$ are all isomorphic;
    \item $\mathcal{X}_i\to C$ and $\mathcal{X}_j \to C$ are not isomorphic when $i\neq j$.
\end{enumerate} 
\end{proposition}

The main ingredient is that birational pairs of smooth irreducible symplectic varieties are non-separated points in the moduli space. More precisely, we have the following  result, which is a mild strengthening of \cite[Theorem 4.6]{Huybrechts}, \cite[Proposition 2.1]{R14} and \cite[Theorem 6.16]{BL22}.

\begin{proposition}\label{prop:def-bir}
 Let $X$ be a $\QQ$-factorial terminal irreducible symplectic variety with $b_2(X) \geq 4$, and $L$ a very ample line bundle on $X$.
There exists a smooth integral complex pointed curve $(C,0)$   such that for any birational map $\Psi \colon X \dashrightarrow X'$ to a $\QQ$-factorial irreducible symplectic variety $X'$ with terminal singularities, there are locally trivial families of primitive symplectic varieties with Picard number one at the geometric generic fibers:
 \[
    \pi_1\colon\mathcal{X}_1\to C, \quad \pi_2\colon \mathcal{X}_2\to C
 \] 
 that are isomorphic over $C^* =C \setminus \{ 0\}$,  and endowed with line bundles $\mathcal{L}_{i}$ on $\mathcal{X}_i$ with 
 \[
 \cL_1|_{\cX_{C^*}} \simeq \cL_2|_{\cX_{C^*}},\quad 
 \mathcal{L}_{1}|_{\pi_{1}^{-1}(0)} \simeq L, \text{ and } \mathcal{L}_{2}|_{\pi_{2}^{-1}(0)} \simeq \Psi_*(L).
 \]
\end{proposition}
\begin{proof}
The statement follows from the argument in \cite[Proposition 2.1]{R14}, which we sketch here.
Denote by $\Lambda$ the Beauville–Bogomolov lattice $\rH^2(X,\ZZ)_{\tf}$. There exists a universal deformation $$\pi \colon \mathscr{X} \to \Def^{\lt} (X),$$ where $\Def^{\lt} (X)$ is the Kuranishi space of deformations of $X$. By \Cref{prop:localTorelliThm} (or \cite{Beauville83localtorelli}), there is a period map $\cP \colon \Def^{\lt}(X) \to \Omega_{\Lambda}$,
which is a local isomorphism. The birational transformation $\Psi$ induces a Hodge isometry 
\[
\Psi^* \colon \rH^2(X', \ZZ)_{\tr} \xrightarrow{\sim} \rH^2(X,\ZZ)_{\tr}.
\]
Any $\Lambda$-marking 
$\varphi \colon \rH^2(X,\ZZ)_{\tf} \to \Lambda$ induces a $\Lambda$-marking $\varphi' \coloneqq \varphi \circ \Psi^*$ on $X'$. Under this marking, one can also define a period map $\cP' \colon \Def^{\lt}(X') \to \Omega_{\Lambda}$ and $\cP(0) = \cP'(0)$ in $\Omega_{\Lambda}$.

Denote $\Def^{\lt}(X, L)$ for the Kuranishi space of deformations of pairs $(X,L)$. It is a closed subspace of 
$\Def^{\lt} (X)$ of dimension $b_2 -3$ given by a smooth hypersurface by \cite[1.14]{Huybrechts}.
Let $i \colon X \hookrightarrow \PP^n$ be an embedding associated with $L$,
and $\mathbf{Hilb}_{\PP^n}^X$ the Hilbert scheme parametrizing closed subschemes in $\PP^n$ which are deformation equivalent to $X$. 
Up to shrinking $\Def^{\lt}(X, L) $,  there exists  an analytic open subset
$W^{\an} \subseteq \mathbf{Hilb}_{\PP^n}^{X,\an}$ with a proper surjective map of complex analytic spaces
\[
W^{\an} \rightarrow \Def^{\lt} (X,L).\]
One can find a smooth integral complex pointed curve $(C,0) $ together with a morphism $g:C\to %
{\mathbf{Hilb}_{\PP^n}^{X}}
$ 
with $g(0)=X$. 
 This gives a family
\[
\pi_{C} \colon \mathcal{X}_{C} \rightarrow C
\]
by pulling back the universal family on $\mathbf{Hilb}_{\PP^n}^{X}$. We may assume that $\pi_{C}$ is a family of smooth irreducible symplectic varieties.
Moreover, we may assume
that $\mathcal{X}_{C,t}$ is of Picrad rank 1
 for general $t\in C$.

We put $C^{\an}:= g^{-1} (W^{\an})$, and let $\pi_{C^{\an}} \colon \mathcal{X}_{C^{\an}} \rightarrow C^{\an}$ be the restriction of $\pi_{C}$.
Let $L ' \subset \Pic (X')$ be $\Psi_*(L)$.
Note that, $L'$ is not necessarily a polarization (it is a polarization if and only if $\Psi$ is an isomorphism). We can identify $\Def^{\lt}(X, L)\cong \Def^{\lt}(X',L')$  via the period maps $\cP'\circ \cP^{-1}$ (after possibly a shrinking). 
 Note that, under this identification, one get another family $\pi' \colon \mathcal{X}' \rightarrow \Def^{\lt}(X, L)$ whose central fiber over $0$ is $X'$.
 Let $\pi'_{C^{\an}} \colon \mathcal{X}^{'}_{C^{\an}} \rightarrow C^{\an}$ be the restriction of $\pi'$.
 Note that the fibers of $\pi'_{C^{an}}$ are all projective by the projectivity criterion (\cite[Theorem 2]{HuybrechtsErratum}) since they have positive line bundle.
 The argument in \cite[Claim 2.2]{R14} shows that there is a birational map  
$$\Psi_t \colon \mathcal{X}_{C^{\an},t} \dashrightarrow \mathcal{X}_{C^{\an},t}'$$
for general $t\in C^{\an}$ and the specialization of 
the Hodge isometry 
\begin{equation}
    \Psi_t^\ast:\rH^2(\mathcal{X}'_{C^{\an},t}, \ZZ)\cong \rH^2(\mathcal{X}_{C^{\an},t},\ZZ)
\end{equation}
to the special fiber is $\Psi^\ast: \rH^2(\mathcal{X}'_0,\ZZ)\cong \rH^2(\mathcal{X}_0,\ZZ)$. 
Moreover, by the argument after \cite[Claim 2.2]{R14}, $\mathcal{X}|_{C^{\an, \ast}}$ is isomorphic to $\mathcal{X}'|_{C^{\an,\ast}}$  after shrinking $C^{\an}$. 

Define a complex manifold $\mathcal{X}'_{C}$ by gluing $\mathcal{X}'_{C^{\an}}$ into $\mathcal{X}_{C} \setminus X$ along the isomorphism $\mathcal{X}_{C^{\an, \ast}} \simeq \mathcal{X}'_{C^{\an, \ast}}$.
Then $\mathcal{X}'_{C}$ is a Zariski open subset of $\overline{\mathcal{X}'_{C}}$ that is a Moishezon space obtained by gluing $\overline{\mathcal{X}_{C}} \setminus X$ with $\mathcal{X}'$ along $\mathcal{X} \setminus X$, where $\overline{\mathcal{X}_C}$ is a closure of $\mathcal{X}$ in any projective embedding.
By Artin's theorem (\cite[Theorem 7.3]{Artin70}), $\mathcal{X}'_{C}$ is an algebraic space.
Let $\mathcal{L}_{1}$ on $\mathcal{X}_{1}$ be the restriction of $\cO(1)$, which is a universal line bundle over $\mathbf{Hilb}_{\PP^n}^X$.
Let $\mathcal{X}_{2}$ be the closure of the restriction of $\mathcal{L}_{1}|_{\mathcal{X}_{1} \setminus X}$ to $\mathcal{X}_{2} \setminus X$.
Then by construction, $\mathcal{L}_{2}$ specializes to $L'$.
It finishes the proof.
\end{proof}

\subsection*{Proof of Proposition \ref{prop:eg}}

Here we give a  construction based on the example given in \cite{HT10}, which has infinitely many distinct flops $\Psi_i\in \Bir(X)$.

Let  $Y\subseteq \PP^5$ be a smooth cubic fourfold containing a cubic scroll $S$. Set
\[
\mathrm{Hdg}(Y) = \rH^{2,2}(Y,\CC) \cap \rH^4(Y,\ZZ).
\]
Choose $Y$ to be general such that $\mathrm{Hdg}(Y)$ is generated by the  classes $[H^2]$ and $[S]$. Let $X$ be the Fano variety of lines in $Y$. The incidence correspondence induces a map
\[
  \alpha \colon \rH^4(Y,\ZZ)\to \rH^2(X,\ZZ).
\]
Set $L=\alpha([H^2])$ and $\ell=\alpha([S])$. Note that $L$ is an ample class on $X$; in fact it is proportional to the Pl\"ucker polarization. Then the Néron--Severi lattice of $X$ under the Beauville--Bogomolov form is given by
\begin{equation*}
 \NS(X) =\left(
\begin{array}{c|ccc}   & L & \ell  
\\ \hline  L  &  6 & 6  \\ \ell & 6  & 2 
 \\\end{array}\right).
\end{equation*}
See \cite[\S 7]{HT10}.
There are no isotropic integral classes and $(-2)$-classes in $\NS(X)$. 
By \cite[Proposition 7.2]{HT10}, the nef cone $\Nef(X) \subseteq \NS(X)_{\RR}$ is 
\[
\Conv(7L - 3\ell, L + 3\ell).
\]
As shown in \cite[Theorem 7.4]{HT10}, there is a flop $\Psi\in \Bir(X)$ of infinite order, whose action on $\NS(X)$ is 
\begin{equation}
\begin{aligned}
      \Psi^\ast(L)&=11L-6\ell\,,  \\
      \Psi^\ast(\ell)& =2L-\ell 
\end{aligned}
\end{equation}
So we can apply Proposition \ref{prop:def-bir} to the self birational maps for all $i$, 
\[
\Psi_i\coloneqq \underbrace{ \Psi\circ  \cdots \circ \Psi}_{i \text{ times}} \colon X \dashrightarrow X
\]
to obtain infinitely many algebraic families $\mathcal{X}_i\to C$. For any $i \neq j$, we have 
\begin{equation}
\label{eqn:AmpleDisjoint}
    \Psi_i(\Amp(X)) \cap \Psi_j(\Amp(X)) = \varnothing
\end{equation}
inside the movable cone $\overline{\Mov}(X)$.

Now we show that $\cX_i\to C$ are distinct families.  Assume, for contradiction, that there exists an isomorphism of families
\[
\begin{tikzcd}
 \cX_i \ar[r,"\phi","\sim"'] \ar[d] & \cX_j \ar[d] \\  
 C \ar[r, "\id_C"] & C
\end{tikzcd}
\]
for some $i\neq j$. Since the generic fiber has Picard number one,  by restricting to the special fiber, one can obtain
an  automorphism \( \phi_0: X \to X \)
whose action on $\NS(X)$ satisfies
\[
\Psi_{i,*}(L) = \Sp_i(\phi_*(\cL_{j,\overline{\eta}})) = \phi_{0,*}\Sp_j(\cL_{j,\overline{\eta}}) = \phi_{0,*}\Psi_{j,\ast}(L),
\]
where $\Sp_i$ and $\Sp_j$ are specialization of Néron--Severi lattices along these two families respectively. Since $\phi_{0, *}$  clearly preserves the decomposition of the movable cone of $X$ into ample chambers, this contradicts \eqref{eqn:AmpleDisjoint}.

\section*{Declaration}
\subsection*{Funding}L.~Fu is supported by the University of Strasbourg Institute for Advanced Study (USIAS), by the Agence Nationale de la Recherche (ANR) under projects ANR-20-CE40-0023 and ANR-24-CE40-4098, and by the CNRS project \textit{International Emerging Actions} (IEA). Z.~Li is supported by NSFC grant (No. 12425105, No.  and No. 12171090) and Shanghai Pilot Program for Basic Research (No. 21TQ00). Zhiyuan is also a member of LMNS. 
T.~Takamatsu is supported by JSPS KAKENHI Grant Numbers JP22KJ1780 and JP25K17228.
H.~Zou is supported by the Deutsche Forschungsgemeinschaft (DFG, German Research Foundation) – Project-ID 491392403 – TRR 358.
\subsection*{Conflict of interest}The authors have no competing interests to declare that are relevant to the content of this article.

\bibliography{PointedShaf}
\bibliographystyle{amsplain}
\end{document}